\documentclass[12pt]{amsart}

\usepackage[margin=1in]{geometry}
\usepackage{setspace}
\usepackage{amsmath}
\usepackage{amsxtra}
\usepackage{amscd}
\usepackage{amsthm}
\usepackage{amssymb}
\usepackage[normalem]{ulem}
\usepackage{comment}
\usepackage[dvipsnames]{xcolor}
\usepackage{tikz}
\usetikzlibrary{matrix,arrows,decorations.pathmorphing,automata, matrix, positioning, calc, shapes.multipart}
\usepackage{graphicx}
\usepackage{hyperref}
\usepackage[shortlabels]{enumitem}
\usepackage{float}
\usepackage{ytableau}

\definecolor{LightGray}{gray}{0.8}

\numberwithin{equation}{section}

\theoremstyle{definition}
\newtheorem{theorem}[equation]{Theorem}
\newtheorem{lemma}[equation]{Lemma}
\newtheorem{proposition}[equation]{Proposition}
\newtheorem{corollary}[equation]{Corollary}

\newtheorem{example}[equation]{Example}

\newtheorem{conjecture}[equation]{Conjecture}
\newtheorem{convention}[equation]{Convention}
\newtheorem{remark}[equation]{Remark}

\newcommand\TikZ[1]{\begin{matrix}\begin{tikzpicture}#1\end{tikzpicture}\end{matrix}}

\newcommand{\hstar}{\mathfrak{h}^*}

\newcommand{\emp}{\emptyset}

\newcommand{\sle}{\widehat{\mathfrak{sl}}_e}

\newcommand{\slinf}{\mathfrak{sl}_\infty}

\newcommand{\triv}{\mathsf{triv}}

\newcommand{\ct}{\mathrm{ct}}
\newcommand{\cct}{\widetilde{\ct}}

\newcommand{\ul}{\underline}
\newcommand{\cA}{\mathcal{A}}
\newcommand{\cF}{\mathcal{F}}

\newcommand{\cO}{\mathcal{O}}
\newcommand{\cP}{\mathcal{P}}
\newcommand{\cU}{\mathcal{U}}

\newcommand{\EE}{\mathsf{E}}
\newcommand{\FF}{\mathsf{F}}

\newcommand{\Res}{\operatorname{Res}}

\newcommand{\Ores}{{^\cO\Res}}

\newcommand{\la}{\lambda}

\usepackage[foot]{amsaddr}

\title[Unitary RCA representations and crystals in type B]{Unitary representations of type B rational Cherednik algebras and crystal combinatorics}
\author{Emily Norton}
\address{Fachbereich Mathematik, TU Kaiserslautern, Gottfried-Daimler-Str. 48, 67663 Kaiserslautern, Germany.}
\email{norton@mathematik.uni-kl.de}

\begin{document}

\begin{abstract}
We compare crystal combinatorics of the level $2$ Fock space with the classification of unitary irreducible representations of type $B$ rational Cherednik algebras to study how unitarity behaves under parabolic restriction. First, we show that any finite-dimensional unitary irreducible representation of such an algebra is labeled by a bipartition consisting of a rectangular partition in one component and the empty partition in the other component. This is a new proof of a result that can be deduced from theorems of Montarani and Etingof-Stoica. 
Second, we show that the crystal operators that remove boxes preserve the combinatorial conditions for unitarity, and that the parabolic restriction functors categorifying the crystals send irreducible unitary representations to unitary representations. 
Third, we find the supports of the unitary representations, showing that all possible supports allowed by the first and second results do in fact occur.
\end{abstract}

\maketitle

\section{Motivation and summary of results}  
\subsection{Motivation} The unitary representations in the category $\cO_c(W)$ of a rational Cherednik algebra associated to a complex reflection group $W$ were first introduced and studied by Etingof and Stoica, who classified the unitary irreducible representations for rational Cherednik algebras of symmetric groups $S_n$ \cite{EtingofStoica}. This classification was extended by Griffeth to the wreath product groups $W=G(\ell,1,n)$, which include $S_n$ as $G(1,1,n)$ and the Weyl group $B_n$ as $G(2,1,n)$ \cite{Griffeth}. The classification for $G(\ell,1,n)$ is inductive and complicated, but for $B_n$ there exist closed conditions for unitarity in terms of contents of certain boxes in bipartitions thanks to \cite[Corollaries 8.4 and 8.5]{Griffeth}. The goal of this paper is to study how these conditions interact with the crystal combinatorics of level $2$ Fock spaces, which controls the cuspidal data of the irreducible representations in $\cO_c(B_n)$ for appropriate parameters $c$.

Unitary representations stand out from the rest of the representations in $\cO_c(W)$ by their relative tractability. In many cases they are known to possess closed-form, multiplicity-free character formulas in terms of a saturated subset of the poset of lowest weights \cite{Ruff}, \cite{BNS}, \cite{FGM}. 
These formulas are known in many cases to admit a combinatorial interpretation via affine type $A$ alcove geometry \cite{BNS}.
It is often possible to write explicit bases for the unitary representations or for the corresponding representations of closely related algebras such as the Morita equivalent diagrammatic Cherednik algebra or the Hecke algebra.\footnote{This has been done in special cases in type $G(\ell,1,n)$ using a variety of methods and perspectives. When a standard module in $\cO_c(G(\ell,1,n))$ is diagonalizable with respect to a certain commutative subalgebra, \cite{GriffethCali} completely describes the submodule structure of that standard module. This implies a basis of its irreducible quotient. Some unitary representations fall under this umbrella, but most do not. Similarly, using these methods in type $A$ the recent paper of \cite{GSV} shows how these bases arise from parabolic Hilbert schemes. In \cite{BNS}, we construct bases of the unitary modules of Hecke and diagrammatic Cherednik algebras of types $A$ and, for certain parameters, $G(\ell,1,n)$.} 

At the categorical level, the unitary irreducible representations in type $A$, and for certain parameters, type $G(\ell,1,n)$,
have been proven to possess resolutions by complexes of direct sums of standard modules \cite{BEG}, \cite{BNS}, \cite{BHN}. 
Such resolutions, known in the literature as BGG resolutions, categorify character formulas via the Euler characteristic of the complex.\footnote{The BGG resolution of $L(\triv)\in\cO_{\frac{1}{h}}(W)$ for $W$ a real reflection group, where $\triv $ is the trivial representation of $W$ and $h$ is the Coxeter number of $W$, was constructed in \cite{BEG}. The BGG resolutions of all unitary irreducible representations in type $A$ were conjectured in \cite{BZGS} and this conjecture was proved in \cite{BNS}, where resolutions of unitary irreducible representations were also constructed in type $G(\ell,1,n)$ for specific types of parameters, providing the first known character formulas of those unitary modules.} Moreover, one might consider not just a unitary representation by itself, but the whole Serre subcategory or subquotient category determined by the saturated subset of the poset of lowest weights appearing in its character. In type $A$, this subcategory of $\cO_{\frac{1}{e}}(S_n)$ is equivalent to a truncation of an Elias-Williamson category of parabolic Soergel bimodules of type $A_{e'}\subset \widetilde{A}_{e'}$ for some $e'\leq e$ \cite{BCH}.
  
In type $A$, the BGG resolutions of unitary modules are invariant under ``runner removal Morita equivalences" \cite{ChuangMiyachi}, equivalently, under dilation of $e$-alcoves in the affine type $A$ alcove geometry determining the maps in the complex \cite{BowmanCox}. Consequently, the BGG resolutions and hence the character formulas of all unitary irreducible modules for type $A$ Cherednik algebras at all parameters are actually classified by a subset of them, the unitary modules $L_{\frac{1}{e}}(e^m)$ labeled by rectangles with $e$ columns, where $c=\frac{1}{e}$ is the parameter for the Cherednik algebra $H_c(S_n)$ and $n=em$ for some $m\geq 1$ \cite{BNS}. All other characters and resolutions of unitary irreducible representations in type $A$ are simply obtained by relabeling partitions via an easy combinatorial procedure. The rectangles $(e^m)$ have an interesting interpretation in terms of the combinatorics of certain simply directed graphs called crystals whose vertices are partitions.  

This hints that it may be interesting to study the interaction of the combinatorial conditions for unitarity with crystal combinatorics on higher level Fock spaces. 
 In this paper, we carry out this comparison for the rational Cherednik algebra of $B_n$ and the level $2$ Fock space.  Perhaps the patterns revealed in this study may lead to a more conceptual description of the unitary irreducible representations in $\cO_c(G(\ell,1,n))$ at parameters $c$ such that $\bigoplus_n\cO_c(G(\ell,1,n)$ categorifies a level $\ell$ Fock space. We also hope it will help to suggest the different types of alcove geometries we need to consider to give an attractive solution to the problem of constructing BGG resolutions of unitary representations in $\cO_c(B_n)$, and ultimately in $\cO_c(G(\ell,1,n))$.

\subsection{The finite-dimensional unitary representations}\label{First theorem} Finite-dimensional representations play the role of cuspidals in the representation theory of rational Cherednik algebras. 
Our first result answers the following question about representations in the category $\cO$ of rational Cherednik algebras associated to the Weyl group $B_n$: which unitary representations of these algebras are also finite-dimensional? To give a precise combinatorial answer in terms of the labeling of irreducible representations by bipartitions of $n$, we use  a criterion by Etingof and Stoica in terms of the $c$-function \cite{EtingofStoica}, the classification of unitary representations by Griffeth \cite{Griffeth}, and the combinatorial description of finite-dimensional representations in terms of source vertices of crystals on a level $2$ Fock space due to Shan, Vasserot, Losev, Gerber, and the author \cite{Shan},\cite{ShanVasserot},\cite{Losev},\cite{Gerber1},\cite{Gerber2},\cite{GerberN}. Let $(e,\ul{s})$ be a parameter for a level $2$ Fock space, so $e\in\mathbb{Z}_{\geq 2}$, and $\ul{s}=(s_1,s_2)\in\mathbb{Z}^2$. To the parameter $(e,\ul{s})$ we can associate a rational Cherednik algebra $H_{e,\ul{s}}(B_n)$ and its category $\cO$ of representations, denoted $\cO_{e,\ul{s}}(n)$. 

\begin{theorem}\label{unitary+fd} Let $\ul{\lambda}$ be a bipartition and let $(e,\ul{s})\in \mathbb{Z}_{\geq 2}\times \mathbb{Z}^2$ be a parameter for the Fock space. The irreducible representation $L(\ul{\lambda})\in\cO_{e,\ul{s}}(n)$ is both unitary and finite-dimensional if and only if (i) $\ul{\lambda}=(\lambda,\emptyset)$ or $\ul{\lambda}=(\emptyset,\lambda)$ for some rectangle $\lambda$, and (ii) the parameter $(e,\ul{s})$ satisfies the following conditions depending on $\ul{\lambda}$, where $r$ is the number of rows and $q$ the number of columns of $\lambda$:
\begin{enumerate} 
\item if $\ul{\lambda}=(\lambda,\emptyset)$ then $r-q=s-e$;
\item if $\ul{\lambda}=(\emptyset,\lambda)$ then $r-q=-s$.
\end{enumerate}
\end{theorem}

It turns out that Theorem \ref{unitary+fd} can be deduced from a result of Montarani on wreath product algebras \cite{Montarani} together with the afore-mentioned criterion of Etingof-Stoica \cite{EtingofStoica}. Montarani studied symplectic reflection algebras associated to wreath products $G(\ell,1,n)$, which encompass the rational Cherednik algebras of $B_n=G(2,1,n)$. Montarani proved that if an irreducible $G(\ell,1,n)$-representation $\ul{\lambda}$ extends to a representation of a rational Cherednik algebra at some parameters then $\ul{\lambda}$ is a rectangle concentrated in a single component, and she gave the formula for the parameters for which it extends depending on the height and width of the rectangle \cite{Montarani}. That result agrees with the formulas in Theorem \ref{unitary+fd}. Etingof and Stoica proved that for the rational Cherednik algebra of a real reflection group $W$, if $L(\ul{\lambda})$ is unitary then $L(\ul{\lambda})$ is in addition finite-dimensional if and only if $L(\ul{\lambda})$ is equal to the irreducible $W$-representation $\ul{\lambda}$ \cite[Proposition 4.1]{EtingofStoica}. This implies Theorem \ref{unitary+fd}. Conversely, the combination of \cite[Proposition 4.1]{EtingofStoica} and Theorem \ref{unitary+fd} implies that for $W=B_n$, $L(\ul{\lambda})=\ul{\lambda}$ if and only if $\ul{\lambda}$ and $(e,\ul{s})$ satisfy the conditions of Theorem \ref{unitary+fd}, yielding a new proof of Montarani's result in type $B_n$.

Our proof is also of independent interest as it uses different techniques -- the combinatorial classification of unitary representations by Griffeth \cite{Griffeth} and the combinatorial classification of finite-dimensional representations arising from categorical affine Lie algebra actions \cite{ShanVasserot},\cite{Shan},\cite{Losev}. In particular, comparing the combinatorial classifications of unitary representations and finite-dimensional representations can in principle be done for $G(\ell,1,n)$ when $\ell>2$ (although the combinatorial classification in \cite{Griffeth} becomes very complicated for $\ell>2$). On the other hand, identifying unitary finite-dimensional representations by checking when $L(\ul{\lambda})=\ul{\lambda}$ is insufficient when $\ell>2$ as $G(\ell,1,n)$ is not a real reflection group and the criterion  \cite[Proposition 4.1]{EtingofStoica} need no longer hold. So we may think of the case $\ell=2$ as a good test case. In particular, Theorem \ref{unitary+fd} provides a sanity check that different combinatorial results on rational Cherednik algebra representations, often proved by intricate calculations and originating from very different methodologies, yield compatible and correct answers.


\subsection{Parabolic restriction and supports of unitary representations in type $B$}
The second and third results of this paper concern the behavior of the unitary representations in $\cO_c(W)$ under the parabolic restriction functor $\Ores^W_{W'}:\cO_c(W)\rightarrow\cO_c(W')$ defined in \cite{BE}, where $W'\subset W$ is a parabolic subgroup. We are interested in the following conjecture of Etingof:

\begin{conjecture}[Etingof's conjecture on restriction of unitary representations]
 If $L$ is a unitary irreducible representation
in $\cO_c(W)$, and $W'\subset W$ is a parabolic
subgroup of $W$, then $\Ores^W_{W'} L$ is unitary.
\end{conjecture}

\noindent Recently, Shelley-Abrahamson extended Etingof's conjecture to consider arbitrary signatures \cite[Section 6.3]{Shelley-Abrahamson}.

Our second theorem is a proof using the combinatorics of the $\sle$-crystal that the parabolic restriction functor $\Ores^n_{n-1}:=\Ores^{B_n}_{B_{n-1}}:\cO_{e,\ul{s}}(n)\rightarrow\cO_{e,\ul{s}}(n-1)$ takes unitary irreducible representations to unitary representations, as well as the analogous statement for the functors corresponding to the operators in the $\slinf$-crystal that remove $e$-strips.\footnote{We note that in the case that $L(\ul\la)$ has full support, which is the case if and only if $\ul\la^t$ has depth $n=|\ul\la|$ in the $\sle$-crystal on $\cF_{e,\ul{s}}^2$, then the first statement of Theorem \ref{Main2} should follow from the results of \cite{Shelley-Abrahamson}. We thank Stephen Griffeth for pointing this out. However, there are other unitary representations floating around that are not sent to $0$ by $\Ores^n_{n-1}$ but that don't have full support either. In advance of doing the work we can't say too much about this, so we don't try to take any shortcuts.} 

\begin{theorem}\label{Main2} Let $L(\ul\la)\in\cO_{e,\ul{s}}(n)$ be a unitary irreducible representation. Then $\Ores^n_{n-1}L(\ul\la)$ is unitary or $0$. Moreover, if $\ul\la^t$ has nonzero depth in the $\slinf$-crystal, then $\Ores^{B_n}_{B_{n-e}\times S_e}L(\ul\la)$ is a unitary irreducible representation, and $\Ores^n_{n-1}L(\ul\la)=0$.
\end{theorem}

\noindent By standard properties of unitary modules, $\Ores^n_{n-1}L(\ul{\lambda})$ is then semisimple, but our argument goes the other way around: semisimplicity is part and parcel of the proof. We show that each direct summand $\mathrm{E}_i$ of the functor $\Ores^{n}_{n-1}$ sends a unitary $L(\ul{\lambda})$ to an irreducible representation or to $0$, and we show that $\tilde{e}_i$, the $\sle$-crystal operator sending $\ul{\lambda}^t$ to $\ul{\mu}^t$ if $L(\ul{\mu})$ is in the head of $\mathrm{E}_i L(\ul{\lambda})$, sends bipartitions labeling unitary irreducible representations to bipartitions labeling unitary irreducible representations. These two steps combined imply that $\Ores^n_{n-1}$ preserves unitarity. Along with the corresponding statements for the $\slinf$-crystal and the functor associated to it, these steps are carried out in Propositions \ref{crystal ops unitaries}, \ref{weak res unitaries semisimple}, and \ref{slinf res}. 

Our methods are combinatorial. We never interact with the definition of unitary representation, we only use the combinatorial classification of bipartitions labeling irreducible representations that are unitary, the combinatorial rules governing the two crystals, and some very easy representation-theoretic consequences of the categorical actions on $\bigoplus\limits_n\cO_{e,\ul{s}}(n)$ connecting the two crystals to parabolic restriction functors \cite{Shan},\cite{ShanVasserot},\cite{ChuangRouquier}.
As \cite[Corollaries 8.4 and 8.5]{Griffeth} give a case-by-case classification of bipartitions labeling unitary representations, our proofs are an exhausting case-by-case analysis as well. We have tried to include some pictures and examples along the way to improve the readability.

As a corollary of the work done in the proof of Proposition \ref{crystal ops unitaries} together with Theorem \ref{unitary+fd}, we can classify the supports of all unitary irreducible representations in $\cO_{e,\ul{s}}(n)$. The module $L(\ul\la)$ is a module over the subalgebra $\mathbb{C}[x_1,\ldots, x_n]$ of the rational Cherednik algebra of $B_n$. The support of $L(\ul\la)$ is defined to be its support as a $\mathbb{C}[x_1,\ldots, x_n]$-module. As explained in \cite{Losev}, the support of $L(\ul\la)$ is of the form $B_n\mathfrak{h}^{W'}$ where $\mathfrak{h}\cong \mathbb{C}^n$ is the reflection representation of $B_n$ and $W'\subseteq B_n$ is a subgroup of the form $W'=B_{n'}\times S_{e}^{ m}\times S_1^{ p}\cong B_{n'}\times S_e^{m}$ such that $n=n'+em+p$. The parabolic subgroup $W'$ is then minimal with respect to the property that $\Ores^{B_n}_{W'}L(\ul\la)\neq 0$, and $\Ores^{B_n}_{W'}L(\ul\la)$ is a finite-dimensional representation in $\cO_c(W')$ \cite{BE}.\footnote{The additional data of the isotypy class $L(\ul\mu)$ of $\Ores^{B_n}_{W'}L(\ul\la)$ together with $W'$ gives the \textit{cuspidal support} of $L(\ul\la)$ as the pair $(L(\ul\mu),W')$.} Losev showed that the number $m$ of copies of $S_e$ in $W'$ is equal to the depth of $\ul\la^t$ in the $\slinf$-crystal on $\cF_{e,\ul{s}}^2$, and the number $p$ of copies of $S_1\cong \{1\}$ in $W'$ is equal to the depth of $\ul\la^t$ in the $\sle$-crystal on $\cF^2_{e,\ul{s}}$ \cite{Losev}. 

\begin{corollary}\label{Main3}
The set of parabolic subgroups $W'\subseteq B_n$ such that $B_n\mathfrak{h}^{W'}$ is the support of some unitary representation $L(\ul\la)\in\cO_{e,\ul{s}}(n)$ is given by:
\begin{itemize}
\item $W'=\{1\}$,
\item $W'=S_e^m$ if and only if $n=em$ for some $m\in\mathbb{N}$,
\item for every $q,r\in\mathbb{N}$ such that $rq\leq n$ and $r-q=s-e$, $W'=B_{rq}$,
\item for every $q,r\in\mathbb{N}$ such that $rq+em=n$ for some $m\in\mathbb{N}$ and $r-q=s-e$, $W'=B_{rq}\times S_e^m$,
\item for every $q,r\in\mathbb{N}$ such that $rq\leq n$ and $r-q=-s$, $W'=B_{rq}$,
\item for every $q,r\in\mathbb{N}$ such that $rq+em=n$ for some $m\in\mathbb{N}$ and $r-q=-s$, $W'=B_{rq}\times S_e^m$.
\end{itemize}
\end{corollary}


\section{Background and notation}
\subsection{Bipartitions, the level $2$ Fock space, and the category $\cO_{e,\ul{s}}(n)$}
A partition is a finite, non-increasing sequence of positive integers: $\lambda=(\lambda_1,\lambda_2,\dots,\lambda_r)$ such that $\lambda_1\geq \lambda_2\geq \dots\geq \lambda_r> 0$.  We use the notation $\emptyset$ for the empty sequence $()$ when $r=0$. Let $\mathcal{P}=\{\lambda\mid \lambda\hbox{ is a partition}\}$ be the set of all partitions. We use the notation $\emptyset$ for the empty sequence $()$ when $r=0$. For a partition $\lambda$ we write $|\lambda|=\sum_{i=1}^r\lambda_i$ and we identify $\lambda$ with its Young diagram, the upper-left-justified array of $|\lambda|$ boxes in the plane with $\lambda_1$ boxes in the top row, $\lambda_2$ boxes in the second row from the top,..., $\lambda_r$ boxes in the $r$'th row from the top. If a box of $\lambda$ is in row $x$ and column $y$ then its content is defined to be $\mathrm{ct}(b):=y-x$. A partition $\lambda$ is a rectangle if and only if $\lambda_1=\lambda_2=\dots=\lambda_r$. The transpose of $\lambda$, denoted $\lambda^t$, is defined to be the partition whose Young diagram has $\lambda_1$ boxes in the leftmost column, $\lambda_2$ boxes in the second column from the left, and so on. 

Let $B_n$ be the type $B$ Weyl group of rank $n$, also known as the hyperoctahedral group or the complex reflection group $G(2,1,n)$. Then $B_n\cong \left(\mathbb{Z}/2\mathbb{Z}\right)^n\rtimes S_n$ and the irreducible representations of $B_n$ over $\mathbb{C}$ are labeled by bipartitions of $n$ defined as $$\mathcal{P}^2(n):= \{\ul{\lambda}=(\lambda^1,\lambda^2)\mid \;\lambda^1,\lambda^2\in\mathcal{P}\hbox{ and }|\lambda^1|+|\lambda^2|=n\}.$$
We call $\lambda^1$ and $\lambda^2$ the components of $\ul{\lambda}$. Set $$\mathcal{P}^2=\bigcup_{n\geq 0}\mathcal{P}^2(n)$$ and fix $e\in\mathbb{N}_{\geq 2}$ and $\ul{s}=(s_1,s_2)\in\mathbb{Z}^2$. 

The level $2$ Fock space $$\cF^2_{e,\ul{s}}=\bigoplus_{\ul{\mu}\in\mathcal{P}^2}\mathbb{C}|\ul{\mu},\ul{s}\rangle$$ depending on the parameter $(e,\ul{s})$ has formal basis given by {\em charged bipartitions} $|\ul{\mu},\ul{s}\rangle$, which are simply bipartitions $\ul{\mu}$ with the additional data of the charge $\ul{s}$. A charged bipartition $|\ul{\mu},\ul{s}\rangle$ may be identified with the pair of Young diagrams for $\mu^1$ and $\mu^2$ with the additional data that each box $b\in\ul{\mu}$ is filled with its {\em charged content}, defined as $\widetilde{\mathrm{ct}}(b)=s_j+y-x$, where $b$ belongs to row $x$ and column $y$ of $\mu^j$ for $j\in\{1,2\}$. Let $s=s_2-s_1$. The algebra we will associate to the level $2$ Fock space depends only on $e$ and $s$, so unless otherwise noted we will always take $\ul{s}=(0,s)$ for some $s\in\mathbb{Z}$. 

\begin{convention}\label{transpose convention}
To $\ul{\lambda}\in\mathcal{P}^2$ we will associate the charged bipartition $|\ul{\lambda}^t,\ul{s}\rangle\in\cF^2_{e,\ul{s}}$.
\end{convention}

The rational Cherednik algebra is a deformation of $\mathbb{C}[x_1,\dots,x_n,y_1,\dots,y_n]\rtimes \mathbb{C} B_n$ with multiplication depending on a pair of parameters $(c,d)\in \mathbb{C}^2$ \cite{EtingofGinzburg, Griffeth}. We write $H_{e,\ul{s}}(B_n)$ for the rational Cherednik algebra depending on parameters $(c,d)$ as defined in \cite[Equation (2.12)]{Griffeth}, where $c$ and $d$ are determined from the Fock space parameter $(e,\ul{s})$ by the formulas \cite{ShanVasserot}:
$$c=\frac{1}{e},\qquad d=-\frac{1}{2}+\frac{s}{e}$$ where as above, $\ul{s}=(s_1,s_2)$ and $s=s_2-s_1$. Thus, the rational Cherednik algebra determined from the Fock space parameter $(e,\ul{s})$ is the same as the rational Cherednik algebra determined from the Fock space parameter $(e,\ul{s}+(a,a))$ for any $a\in\mathbb{Z}$. This justifies our fixing $\ul{s}=(0,s)$ once and for all.

The algebra $H_{e,\ul{s}}(B_n)$ has a category $\cO_{e,\ul{s}}(n)$ of representations which in particular contains all finite-dimensional representations \cite{GGOR}. The irreducible representations in $\cO_{e,\ul{s}}(n)$ are labeled by $\mathcal{P}^2(n)$ and denoted $L(\ul{\lambda})$ for $\ul{\lambda}\in\mathcal{P}^2(n)$ \cite{GGOR}. The irreducible representation $L(\ul{\lambda})$ comes with a non-degenerate contravariant Hermitian form and $L(\ul{\lambda})$ is called unitary if this form is positive definite \cite{EtingofStoica}. The unitary $L(\ul{\lambda})$ were classified combinatorially by Griffeth \cite[Corollary 8.4]{Griffeth}. In fact, he considers arbitrary parameters $(c,d)\in\mathbb{C}^2$ without any reference to the Fock space. It is known that the resulting category $\cO$ is equivalent to either the one arising from a level $2$ Fock space as above, or to a direct sum of tensor products of category $\cO$'s for rational Cherednik algebras of symmetric groups \cite[Theorem 6.3]{Rouquier}. In this paper, we are focused on the more interesting case, when the Grothendieck group of $\cO$ is a level $2$ Fock space. The other case leaves little room for finite-dimensional unitary modules due to the paucity of finite-dimensional modules in type $A$ (only the row or column can ever be finite-dimensional, and this happens only when the parameter is rational with denominator equal to the number of boxes \cite{BEG}). Similarly, all the questions we are concerned with in this paper -- restriction of the unitary representations, their positions in the crystals on the Fock space and their supports -- are known in type $A$, hence known in the case of a tensor product of type $A$ categories.

The element $\sum_{i=1}^n x_iy_i+y_ix_i\in H_{e,\ul{s}}(B_n)$ acts on the irreducible $\mathbb{C} B_n$-representation $\ul{\lambda}$, the lowest weight space of $L(\ul{\lambda})$, by a scalar $c_{\ul{\lambda}}$. Using the formula given in \cite[Section 5.1]{GGJL} and then correcting for the renormalization introduced in that formula, we have the following formula for $c_{\ul{\lambda}}$ in terms of the parameter $(e,\ul{s})$:
$$c_{\ul{\lambda}}=|\lambda^1|+\frac{s}{e}\left(|\lambda^2|-|\lambda^1|\right)-\frac{2}{e}\sum_{b\in\ul{\lambda}}\mathrm{ct}(b)$$
where the sum runs over all boxes $b$ in $\ul{\lambda}$. The category $\cO_{e,\ul{s}}(n)$ is a lowest weight category\footnote{ i.e. a ``highest weight category" but with lowest instead of highest weights} with order given by $\ul\la<_c\ul\mu$ if and only if $c_{\ul\la}<c_{\ul\mu}$. In particular: if $[\Delta(\ul\la):L(\ul\mu)]>0$ then $c_{\ul\la}<c_{\ul\mu}$; the head of $\Delta(\ul\la)$ is $L(\ul\la)$; and $[\Delta(\ul\la):L(\ul\la)]=1$.

By abuse of notation, write $\lambda$ for the irreducible $B_n$-representation labeled by $\lambda$. Write $\mathfrak{h}^\ast$ for the reflection representation of $B_n$.
\begin{lemma}\cite{EtingofStoica}\label{lemma ES}
A unitary representation $L(\ul{\lambda})$ is finite-dimensional if and only if $L(\ul{\lambda})=\ul{\lambda}$. The statement $L(\ul{\lambda})=\ul{\lambda}$ holds if and only if $c_{\ul{\lambda}}=0$ and for every irreducible constituent $\ul{\mu}$ of the $B_n$-representation $\mathfrak{h}^\ast\otimes\ul{\lambda}$, we have $c_{\ul{\mu}}=1$.
\end{lemma}


The Fock space $\cF^2_{e,\ul{s}}$ comes endowed with the structure of an $\sle$-crystal \cite{JMMO} as well as an exotic structure of an $\slinf$-crystal \cite{ShanVasserot,Losev,Gerber1,Gerber2}, and the combinatorial formulas for both actions depend on $\ul{s}$ and $e$. A crystal is a certain directed graph with at most one arrow between any two vertices. The vertices of both the $\sle$- and the $\slinf$-crystal are $\{|\ul{\lambda}^t,\ul{s}\rangle\mid\ul{\lambda}\in\mathcal{P}^2\}$. We call $|\ul{\lambda}^t,\ul{s}\rangle$ a source vertex in a crystal if it has no incoming arrow. 
These crystals can be derived from a realization of $\cF^2_{e,\ul{s}}$ as the Grothendieck group of $$\cO_{e,\ul{s}}:=\bigoplus_{n\geq 0}\cO_{e,\ul{s}}(n)$$ where $\cO_{e,\ul{s}}(n)$ is the category $\cO$ of the rational Cherednik algebra of $B_n$ with parameters $(c,d)$ determined by the Fock space parameter $(e,\ul{s})$. The $\sle$-crystal on the Fock space arises via a categorical action of $\sle$ on $\cO_{e,\ul{s}}$; the Chevalley generators $f_i$ and $e_i$ act by $i$-induction and $i$-restriction functors, direct summands of the parabolic induction and restriction functors between $\cO_{e,\ul{s}}(n)$ and $\cO_{e,\ul{s}}(n+1)$ \cite{Shan}. Furthermore, there is a Heisenberg algebra action on $\cO_{e,\ul{s}}$ relevant for describing branching of irreducible representations with respect to 
categories $\cO$ attached to
 parabolic subgroups of the form 
$B_m\times S_e^{\times k}$ \cite{ShanVasserot}. The categorical Heisenberg action gives rise to an $\slinf$-crystal structure on $\cF^2_{e,\ul{s}}$ whose arrows add a certain collection of $e$ boxes at a time to a bipartition \cite{Losev}. This $\slinf$-crystal commutes with the $\sle$-crystal \cite{Gerber1}.

\begin{theorem}\cite{ShanVasserot}\label{source vx} The irreducible representation $L(\ul{\lambda})\in \cO_{e,\ul{s}}(n)$ is finite-dimensional if and only if $\ul{\lambda}^t$ is a source vertex for both the $\sle$-crystal and the $\slinf$-crystal on $\cF^2_{e,\ul{s}}$.
\end{theorem}

\noindent In the next two sections, we recall the combinatorial rules for these crystals.

\subsection{The $\sle$-crystal} Let $\mu\in\cP^2$ be a bipartition. A box $b\in\ul\mu$ is said to be removable if $\ul\mu\setminus\{b\}\in\cP^2$. Similarly, an addable box of $\ul\mu$ is a box $a\notin\mu$ such that $\mu\cup\{a\}\in\cP^2$.

The generators $e_i$ and $f_i$ of $\sle$, $i\in\mathbb{Z}/e\mathbb{Z}$, act on $\cO_{e,\ul{s}}$ by biadjoint, exact functors $\EE_i$ and $\FF_i$ that are direct summands of parabolic induction and restriction functors \cite{Shan}. We will only use the functors $\EE_i$, which satisfy:
$$\Ores^{B_n}_{B_{n-1}}\simeq \bigoplus_{i\in\mathbb{Z}/e\mathbb{Z}}\EE_i.$$
 Applying $\EE_i$ to a standard module $\Delta(\ul\la)$, we get a $\Delta$-filtered module $\EE_i\Delta(\ul{\lambda})$ whose standard subquotients are all those $\Delta(\ul{\mu})$ where $|\ul{\mu}^t,\ul{s}\rangle\in\cF_{e,\ul{s}}^2$ is obtained from $|\ul{\lambda}^t,\ul{s}\rangle$ by removing a box of charged content $i\mod e$ (a ``removable $i$-box")\cite{Shan}: $$[\EE_i\Delta(\ul{\lambda})]=\sum_{|\ul{\mu}^t,\ul{s}\rangle=|\ul{\lambda}^t\setminus \{i-\mathrm{box}\},\ul{s}\rangle}[\Delta(\ul{\mu})].$$
If $|\ul{\lambda}^t,\ul{s}\rangle$ has no removable $i$-box then $\EE_i\Delta(\ul{\lambda})=0$. On irreducible modules, $\EE_i$ sends $L(\ul{\lambda})$ either to $0$ or to an indecomposable module $\EE_iL(\ul{\lambda})$ whose head and socle are simple and equal to $L((\tilde{e}_i(\ul{\lambda}^t)^t)$ \cite{Shan}. The charged bipartition $|\tilde{e}_i(\ul{\lambda}^t),\ul{s}\rangle$ is defined as $|\ul\la^t,\ul{s}\rangle$ minus a certain ``good removable $i$-box" when $\EE_iL(\ul\la)$ is nonzero, and $0$ otherwise. 

Let us recall the rule for finding the good removable $i$-box \cite{JMMO}. Order all of the removable and addable $i$-boxes in $|\ul\la^t,\ul{s}\rangle$ according to increasing order of charged content, with a box in component $1$ being considered as bigger than a box of the same charged content in component $2$. This produces a totally ordered list. Then recursively cancel all occurrences of a removable box immediately followed by an addable box in the list. The good removable $i$-box, if it exists, is then the smallest removable $i$-box remaining after no more cancellations are possible. Likewise, the good addable $i$-box is the largest addable $i$-box remaining.
The $\sle$-crystal is then defined to be the simple directed graph with vertices given by the charged bipartitions $\{|\ul\la^t,\ul{s}\rangle\mid\ul\la\in\cP^2\}$ and with an edge $|\ul\mu^t,\ul{s}\rangle\rightarrow|\ul\la^t,\ul{s}\rangle$ if and only if $\tilde{e}_i|\ul{\lambda}^t,\ul{s}\rangle=|\ul{\mu}^t,\ul{s}\rangle$ for some $i\in\mathbb{Z}/e\mathbb{Z}$. As is customary we drop the notation of $\ul{s}$ when it is understood and write $\tilde{e}_i(\ul\la^t)=\ul\mu^t$, but we emphasize that this also depends on $\ul{s}$. If $\tilde{e}_i(\ul\la^t)=\ul\mu^t$ then $\tilde{f}_i(\ul\mu^t)=\ul\la^t$ and $\tilde{f}_i$ adds the good addable $i$-box of $\ul\mu^t$.

\begin{example} Let $e=6$ and $\ul{s}=(0,1)$. Take $\ul{\lambda}=((2^3,1^2),(1^2))$, so $\ul{\lambda}^t=((5,3),(2))$. Then $\ul{\lambda}^t$ has just three removable boxes, and their charged contents mod $6$ are distinct. The removable and addable boxes, colored according to their charged contents mod $6$, are pictured below:
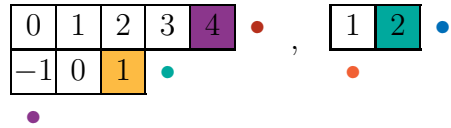
\begin{center}
\ytableausetup{mathmode,boxsize=1.4em}
\begin{ytableau}
0&1&2&3&*(Fuchsia)4&\none[\color{BrickRed}\bullet]\\
-1&0&*(Dandelion)1&\none[\color{Emerald}\bullet]\\
\none[\color{Fuchsia}\bullet]
\end{ytableau}
,\quad
\begin{ytableau}
1&*(Emerald)2&\none[\color{NavyBlue}\bullet]\\
\none[\color{RedOrange}\bullet]
\end{ytableau}
\end{center}
We have $\EE_{\color{Fuchsia}\mathbf{4}}\Delta(\ul{\lambda})=\Delta((2^3,1),(1^2))$ and $(\tilde{e}_{\color{Fuchsia}\mathbf{4}}(\ul{\lambda}^t))^t=((2^3,1),(1^2))$, $\EE_{\color{Emerald}\mathbf{2}}\Delta(\ul{\lambda})=\Delta((2^3,1^2),(1))$ and $\tilde{e}_{\color{Emerald}\mathbf{2}}(\ul{\lambda}^t)=0$ (the addable ${\color{Emerald}2}$-box in the second row of $(\lambda^1)^t$ cancels the removable ${\color{Emerald}2}$-box), $\EE_{\color{Dandelion}\mathbf{1}}\Delta(\ul{\lambda})=\Delta((2^2,1^3),(1^2))$ and $(\tilde{e}_{\color{Dandelion}\mathbf{1}}(\ul{\lambda}^t))^t=((2^2,1^3),(1^2))$.
\end{example}

\begin{example} Let $e=3$, $\ul{s}=(0,1)$, and $\ul\la^t=((2^2,1^2),(2,1^2))$. The removable and addable boxes of charged contents ${\color{CarnationPink}\mathbf{0}}$, ${\color{LimeGreen}\mathbf{1}}$, and ${\color{CornflowerBlue}\mathbf{2}}$ mod $3$ are indicated in the picture below. We have $\tilde{e}_{\color{CarnationPink}\mathbf{0}}(\ul\la^t)=((2^2,1)(2,1^2))$, $\tilde{e}_{\color{LimeGreen}\mathbf{1}}(\ul\la^t)=0$, and $\tilde{e}_{\color{CornflowerBlue}\mathbf{2}}(\ul\la^t)=0$.
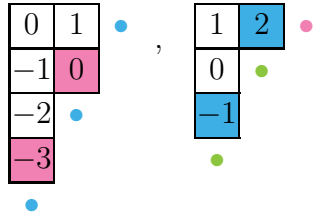
\begin{center}
\begin{ytableau}
0&1&\none [\color{CornflowerBlue}\bullet]\\
-1&*(CarnationPink)0\\
-2&\none[\color{CornflowerBlue}\bullet]\\
*(CarnationPink)-3\\
\none[\color{CornflowerBlue}\bullet]
\end{ytableau}
,\quad
\begin{ytableau}
1&*(CornflowerBlue)2&\none[\color{CarnationPink}\bullet]\\
0&\none[\color{LimeGreen}\bullet]\\
*(CornflowerBlue)-1\\
\none[\color{LimeGreen}\bullet]
\end{ytableau}
\end{center}
\end{example}

We will need two easy lemmas.
\begin{lemma}\label{min cancellation pair} Let $|\ul\la^t,\ul{s}\rangle\in\cF^2_{e,\ul{s}}$.
\begin{enumerate} 
\item Suppose $b$ is a removable box of $(\la^1)^t$ and $a'$ is an addable box of $(\la^2)^t$ such that $\cct(a')=\cct(b)+e$. Let $i=\cct(b)\mod e$. Then $b$ is not a good removable $i$-box.
\item Suppose $b'$ is a removable box of $(\la^2)^t$ and $a$ is an addable box of $(\la^1)^t$ such that $\cct(a)=\cct(b')$. Let $i=\cct(b')\mod e$. Then $b'$ is not a good removable $i$-box.
\end{enumerate}
\end{lemma}
\begin{proof} In the total order on addable and removable $i$-boxes in $|\ul\la^t,\ul{s}\rangle$, $b<a'$ and $b'<a$, and there cannot exist any $i$-box $f$ either addable or removable such that $b<f<a'$ or $b'<f<a$. 
\end{proof}

\begin{lemma}\label{easy i-res} Suppose that for any two removable boxes $b_1,b_2$ of $|\ul\la^t,\ul{s}\rangle$, $\cct(b_1)\neq \cct(b_2)\mod e$. Then $\Ores^n_{n-1}L(\ul\la)$ is semisimple or $0$.
\end{lemma}
\begin{proof} We have $\Ores^n_{n-1}=\bigoplus\limits_{i\in\mathbb{Z}/e\mathbb{Z}}\EE_i$, so it suffices to show that $\EE_iL(\ul\la)$ is either irreducible or $0$ for every $i\in\mathbb{Z}/e\mathbb{Z}$.

By assumption, $|\ul\la^t,\ul{s}\rangle$ has at most one removable $i$-box for any $i\in\mathbb{Z}/e\mathbb{Z}$. If $|\ul\la^t,\ul{s}\rangle$ has no removable $i$-box then $\EE_i$ acts by $0$ on $\Delta(\ul\la)$ and $L(\ul\la)$. So suppose that $|\ul\la^t,\ul{s}\rangle$ has a removable $i$-box $b^t$, corresponding under the transpose map to the box $b\in\ul\la$. Since $b^t$ is the only removable $i$-box, $\EE_i\Delta(\ul\la)=\Delta(\ul\la\setminus b)$. If $b^t$ is not good removable then $\tilde{e}_i(\ul\la^t)=0$, so $\EE_iL(\ul\la)=0$ and there is nothing to show. If $b^t$ is good removable, then $(\tilde{e}_i(\ul\la^t)^t)=\ul\la\setminus b$. By exactness of $\EE_i$ there is a surjective homomorphism $\EE_i\Delta(\ul\la)=\Delta(\ul\la\setminus b)\twoheadrightarrow \EE_i L(\ul\la)$. Pre-composing this with the surjection $\EE_iL(\ul\la)\twoheadrightarrow L(\tilde{e}_i(\ul\la^t)^t)=L(\ul\la\setminus b)$, we obtain a surjection $\Delta(\ul\la\setminus b)\twoheadrightarrow L(\ul\la\setminus b)$ factoring through $\EE_iL(\ul\la)$. 
Since $\Delta(\ul\la\setminus b)$ contains $L(\ul\la\setminus b)$ exactly once as a composition factor, $\EE_i L(\ul\la)$ also contains $L(\ul\la\setminus b)$ exactly once as a composition factor. But the head of $\EE_iL(\ul\la)$ is isomorphic to its socle, and $L(\ul\la\setminus b)$ is the head of $\EE_iL(\ul\la)$. It follows that $\EE_iL(\ul\la)$ is irreducible and equal to  $L(\ul\la\setminus b)$.
\end{proof}

\subsection{The $\slinf$-crystal} A direct combinatorial rule for computing the arrows in the $\slinf$-crystal was given in \cite{GerberN} and builds off of previous results by Gerber \cite{Gerber1},\cite{Gerber2} and Losev \cite{Losev}. The rule uses abacus combinatorics. We define the abacus $\cA(\ul{\lambda})$ via the $\beta$-numbers of the charged bipartition $|((\lambda^1)^t,(\lambda^2)^t),\ul{s}\rangle$ as $\cA(\ul{\lambda}):=\left(\{\beta^1_1,\beta^1_2,\beta^1_3,\dots\},\{\beta^2_1,\beta^2_2,\beta^2_3,\dots\}\right)$, where $\beta^j_i:=(\lambda^j)^t_i+s_j-i+1$ if $i\leq r$ where $r$ is the number of parts of $(\lambda^j)^t$, i.e. the number of columns of $\lambda^j$, and $\beta^j_i=s_j-i+1$ if $i>r$.\footnote{Here as in Convention \ref{transpose convention}, we have taken the transposes of $\lambda^1$ and $\lambda^2$ because the convention of \cite{Griffeth} and \cite{EtingofStoica} for labeling irreducible representations by bipartitions is the component-wise transpose of the convention in \cite{Gerber1},\cite{Gerber2},\cite{GerberN} for the level $2$ Fock space.} It is convenient to visualize $\cA(\ul{\lambda})$ as an array of beads and spaces assembled in two horizontal rows and infinitely many columns indexed by $\mathbb{Z}$; the bottom row corresponds to $\lambda^1$ and the top row to $\lambda^2$, with a bead in row $j$ and column $\beta^j_i$ for each $i\in\mathbb{N}$ and each $j=1,2$, and spaces otherwise. Far to the left the abacus consists only of beads, far to the right only of spaces, so computations always occur in a finite region. The beads in $\cA(\ul{\lambda})$ which have some space to their left correspond to the nonzero columns of $\ul{\lambda}$. We emphasize that the abacus $\cA(\ul\la)$ is really the same thing as the charged bipartition $|\ul\la^t,\ul{s}\rangle$ -- the former is just a visual way of working with the latter. Because we have a running convention of taking the transpose when going from $\cP^2$ to $\cF_{e,\ul{s}}$, we will try to use the phrase and notation ``the abacus $\cA(\ul\la)$" for $\ul\la\in\cP^2$ labeling $L(\ul\la)\in\cO_{e,\ul{s}}$, and to use the phrase ``the abacus associated to $|\ul\la^t,\ul{s}\rangle$" for the same abacus, but now defined with reference to a charged bipartition in $\cF_{e,\ul{s}}$.

\begin{example} Let $\ul{\lambda}=((5,4,2,2),(4,3,2,2))$, $\ul{s}=(0,3)$. Then $\ul{\lambda}^t=((4,4,2,2,1),(4,4,2,1)))$ and we have $\cA(\ul{\lambda})=\left(\{4,3,0,-1,-3,-5,-6,-7,\ldots\},\{7,6,3,1,-1,-2,-3,\ldots\}\right)$, visualized as the following abacus diagram (the beads are understood to extend infinitely to the left, the spaces infinitely to the right):
\[
\TikZ{[scale=.6]
\draw
(9,2)node[fill,circle,inner sep=.5pt]{}
(9,1)node[fill,circle,inner sep=.5pt]{}
(8,2)node[fill,circle,inner sep=.5pt]{}
(8,1)node[fill,circle,inner sep=.5pt]{}
(7,2)node[fill,circle,inner sep=3pt]{}
(7,1)node[fill,circle,inner sep=.5pt]{}
(6,2)node[fill,circle,inner sep=3pt]{}
(6,1)node[fill,circle,inner sep=.5pt]{}
(5,2)node[fill,circle,inner sep=.5pt]{}
(5,1)node[fill,circle,inner sep=.5pt]{}
(4,2)node[fill,circle,inner sep=.5pt]{}
(4,1)node[fill,circle,inner sep=3pt]{}
(3,2)node[fill,circle,inner sep=3pt]{}
(3,1)node[fill,circle,inner sep=3pt]{}
(2,2)node[fill,circle,inner sep=.5pt]{}
(2,1)node[fill,circle,inner sep=.5pt]{}
(1,2)node[fill,circle,inner sep=3pt]{}
(1,1)node[fill,circle,inner sep=.5pt]{}
(0,2)node[fill,circle,inner sep=.5pt]{}
(0,1)node[fill,circle,inner sep=3pt]{}
(-1,2)node[fill,circle,inner sep=3pt]{}
(-1,1)node[fill,circle,inner sep=3pt]{}
(-2,2)node[fill,circle,inner sep=3pt]{}
(-2,1)node[fill,circle,inner sep=.5pt]{}
(-3,2)node[fill,circle,inner sep=3pt]{}
(-3,1)node[fill,circle,inner sep=3pt]{}
(-4,2)node[fill,circle,inner sep=3pt]{}
(-4,1)node[fill,circle,inner sep=.5pt]{}
(-5,2)node[fill,circle,inner sep=3pt]{}
(-5,1)node[fill,circle,inner sep=3pt]{}
(-6,2)node[fill,circle,inner sep=3pt]{}
(-6,1)node[fill,circle,inner sep=3pt]{}
(9,0)node[]{9}
(8,0)node[]{8}
(7,0)node[]{7}
(6,0)node[]{6}
(5,0)node[]{5}
(4,0)node[]{4}
(3,0)node[]{3}
(2,0)node[]{2}
(1,0)node[]{1}
(0,0)node[]{0}
(-1,0)node[]{-1}
(-2,0)node[]{-2}
(-3,0)node[]{-3}
(-4,0)node[]{-4}
(-5,0)node[]{-5}
(-6,0)node[]{-6}
;
}
\]
\end{example}
 
 Following \cite[Definition 2.2]{JaconLecouvey}, given $\cA(\ul{\lambda})$ the abacus associated to $|\ul{\lambda}^t,\ul{s}\rangle$, we say that $\cA(\ul{\lambda})$ has an $e$-period if there exist $\beta^{j_1}_{i_1},\beta^{j_2}_{i_2},\dots,\beta^{j_e}_{i_e}\in\cA(\ul{\lambda})$ such that $j_1\geq j_2\geq\dots\geq j_e$, $\beta^{j_{m+1}}_{i_{m+1}}=\beta^{j_m}_{i_m}-1$ for each $m=1,\dots,e-1$, $\beta^{j_1}_{i_1}$ is in the rightmost column of $\cA(\lambda)$ containing a bead, and for each $m=1,\dots,e$ if $j_m=2$ then $\beta^2_{i_m}$ has an empty space below it. Let $\mathrm{Per}^1=\{\beta^{j_1}_{i_1},\beta^{j_2}_{i_2},\dots,\beta^{j_e}_{i_e}\}\subset\cA(\ul{\lambda})$ denote the $e$-period of $\cA(\ul{\lambda})$ if it exists, and call it the first $e$-period. The $k$'th $e$-period $\mathrm{Per}^k$ of $\cA(\ul{\lambda})$ is defined recursively as the $e$-period of $\left(\dots\left(\cA({\lambda})\setminus \mathrm{Per^1}\right)\setminus\mathrm{Per^2}\dots\right)\setminus\mathrm{Per}^{k-1}$ if it exists and $\mathrm{Per}^1,\dots,\mathrm{Per}^{k-1}$ exist. We
say that $\cA(\ul{\lambda})$ is totally $e$-periodic if $\mathrm{Per}^k$ exists for any $k\in\mathbb{N}$ (see \cite[Definition 5.4]{JaconLecouvey}). This is equivalent to requiring that after replacing all beads with spaces in the first $k$ periods for some $k\in\mathbb{N}$, we are left with the abacus of $(\emptyset,\emptyset)$ for some charge $\ul{s'}$ \cite[Lemma 5.3]{JaconLecouvey}.

\begin{example}\label{exl:totper} Let $\ul{\lambda}=((6,6,6,6),\emptyset)$, $\ul{s}=(0,3)$, and $e=5$. Thus $\ul{\lambda}^t=((4,4,4,4,4,4),\emptyset)$ and the abacus $\cA(\ul{\lambda})$ is drawn below with its first five $5$-periods. We see that $\cA(\ul{\lambda})$ is totally $5$-periodic:
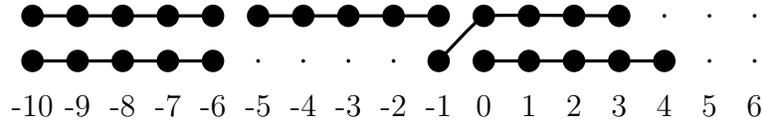
\[
\TikZ{[scale=.6]
\draw
(6,2)node[fill,circle,inner sep=.5pt]{}
(6,1)node[fill,circle,inner sep=.5pt]{}
(5,2)node[fill,circle,inner sep=.5pt]{}
(5,1)node[fill,circle,inner sep=.5pt]{}
(4,2)node[fill,circle,inner sep=.5pt]{}
(4,1)node[fill,circle,inner sep=3pt]{}
(3,2)node[fill,circle,inner sep=3pt]{}
(3,1)node[fill,circle,inner sep=3pt]{}
(2,2)node[fill,circle,inner sep=3pt]{}
(2,1)node[fill,circle,inner sep=3pt]{}
(1,2)node[fill,circle,inner sep=3pt]{}
(1,1)node[fill,circle,inner sep=3pt]{}
(0,2)node[fill,circle,inner sep=3pt]{}
(0,1)node[fill,circle,inner sep=3pt]{}
(-1,2)node[fill,circle,inner sep=3pt]{}
(-1,1)node[fill,circle,inner sep=3pt]{}
(-2,2)node[fill,circle,inner sep=3pt]{}
(-2,1)node[fill,circle,inner sep=.5pt]{}
(-3,2)node[fill,circle,inner sep=3pt]{}
(-3,1)node[fill,circle,inner sep=.5pt]{}
(-4,2)node[fill,circle,inner sep=3pt]{}
(-4,1)node[fill,circle,inner sep=.5pt]{}
(-5,2)node[fill,circle,inner sep=3pt]{}
(-5,1)node[fill,circle,inner sep=.5pt]{}
(-6,2)node[fill,circle,inner sep=3pt]{}
(-6,1)node[fill,circle,inner sep=3pt]{}
(-7,2)node[fill,circle,inner sep=3pt]{}
(-7,1)node[fill,circle,inner sep=3pt]{}
(-8,2)node[fill,circle,inner sep=3pt]{}
(-8,1)node[fill,circle,inner sep=3pt]{}
(-9,2)node[fill,circle,inner sep=3pt]{}
(-9,1)node[fill,circle,inner sep=3pt]{}
(-10,2)node[fill,circle,inner sep=3pt]{}
(-10,1)node[fill,circle,inner sep=3pt]{}
(6,0)node[]{6}
(5,0)node[]{5}
(4,0)node[]{4}
(3,0)node[]{3}
(2,0)node[]{2}
(1,0)node[]{1}
(0,0)node[]{0}
(-1,0)node[]{-1}
(-2,0)node[]{-2}
(-3,0)node[]{-3}
(-4,0)node[]{-4}
(-5,0)node[]{-5}
(-6,0)node[]{-6}
(-7,0)node[]{-7}
(-8,0)node[]{-8}
(-9,0)node[]{-9}
(-10,0)node[]{-10}
;
\draw[very thick] plot [smooth,tension=.1]
coordinates{(4,1)(0,1)};
\draw[very thick] plot [smooth,tension=.1]
coordinates{(3,2)(0,2)(-1,1)};
\draw[very thick] plot [smooth,tension=.1]
coordinates{(-1,2)(-5,2)};
\draw[very thick] plot [smooth,tension=.1]
coordinates{(-6,2)(-10,2)};
\draw[very thick] plot [smooth,tension=.1]
coordinates{(-6,1)(-10,1)};
}
\]
\end{example}

The charged bipartition $|\ul{\lambda}^t,\ul{s}\rangle$ is a source vertex for the $\sle$-crystal on $\mathcal{F}^2_{e,\ul{s}}$ if and only if $\cA(\ul{\lambda})$ is totally $e$-periodic \cite[Theorem 5.9]{JaconLecouvey}. 
In \cite{GerberN} we gave a criterion for checking if $|\ul{\lambda}^t,\ul{s}\rangle$ is a source vertex in the $\slinf$-crystal in terms of a pattern avoidance condition on $\cA(\ul{\lambda})$ \cite[Theorem 7.13]{GerberN}. In particular, any ``log jam" as in Example \ref{exl:totper} where no $e$-period can ``physically move to the left" is always a source vertex in the $\slinf$-crystal.

The rule for edges in the $\slinf$-crystal in a totally $e$-periodic abacus is in terms of shifting $e$-periods to the left and right. In the case that an abacus $\cA$ is not totally $e$-periodic, the rule is more complicated and we refer the reader to \cite{GerberN} for the general case. We now explain the rule in the totally $e$-periodic case. First, we give examples of shifting $e$-periods to the left and right in an abacus.
\begin{example}
Let $\cA$ be the abacus in Example \ref{exl:totper}. The only $e$-period that can be shifted to the right is $\mathrm{Per}^1$. Doing so, we obtain the following abacus:
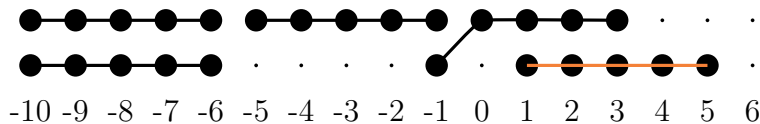
\[
\TikZ{[scale=.6]
\draw
(6,2)node[fill,circle,inner sep=.5pt]{}
(6,1)node[fill,circle,inner sep=.5pt]{}
(5,2)node[fill,circle,inner sep=.5pt]{}
(0,1)node[fill,circle,inner sep=.5pt]{}
(4,2)node[fill,circle,inner sep=.5pt]{}
(4,1)node[fill,circle,inner sep=3pt]{}
(3,2)node[fill,circle,inner sep=3pt]{}
(3,1)node[fill,circle,inner sep=3pt]{}
(2,2)node[fill,circle,inner sep=3pt]{}
(2,1)node[fill,circle,inner sep=3pt]{}
(1,2)node[fill,circle,inner sep=3pt]{}
(1,1)node[fill,circle,inner sep=3pt]{}
(0,2)node[fill,circle,inner sep=3pt]{}
(5,1)node[fill,circle,inner sep=3pt]{}
(-1,2)node[fill,circle,inner sep=3pt]{}
(-1,1)node[fill,circle,inner sep=3pt]{}
(-2,2)node[fill,circle,inner sep=3pt]{}
(-2,1)node[fill,circle,inner sep=.5pt]{}
(-3,2)node[fill,circle,inner sep=3pt]{}
(-3,1)node[fill,circle,inner sep=.5pt]{}
(-4,2)node[fill,circle,inner sep=3pt]{}
(-4,1)node[fill,circle,inner sep=.5pt]{}
(-5,2)node[fill,circle,inner sep=3pt]{}
(-5,1)node[fill,circle,inner sep=.5pt]{}
(-6,2)node[fill,circle,inner sep=3pt]{}
(-6,1)node[fill,circle,inner sep=3pt]{}
(-7,2)node[fill,circle,inner sep=3pt]{}
(-7,1)node[fill,circle,inner sep=3pt]{}
(-8,2)node[fill,circle,inner sep=3pt]{}
(-8,1)node[fill,circle,inner sep=3pt]{}
(-9,2)node[fill,circle,inner sep=3pt]{}
(-9,1)node[fill,circle,inner sep=3pt]{}
(-10,2)node[fill,circle,inner sep=3pt]{}
(-10,1)node[fill,circle,inner sep=3pt]{}
(6,0)node[]{6}
(5,0)node[]{5}
(4,0)node[]{4}
(3,0)node[]{3}
(2,0)node[]{2}
(1,0)node[]{1}
(0,0)node[]{0}
(-1,0)node[]{-1}
(-2,0)node[]{-2}
(-3,0)node[]{-3}
(-4,0)node[]{-4}
(-5,0)node[]{-5}
(-6,0)node[]{-6}
(-7,0)node[]{-7}
(-8,0)node[]{-8}
(-9,0)node[]{-9}
(-10,0)node[]{-10}
;
\draw[very thick,Orange] plot [smooth,tension=.1]
coordinates{(5,1)(1,1)};
\draw[very thick] plot [smooth,tension=.1]
coordinates{(3,2)(0,2)(-1,1)};
\draw[very thick] plot [smooth,tension=.1]
coordinates{(-1,2)(-5,2)};
\draw[very thick] plot [smooth,tension=.1]
coordinates{(-6,2)(-10,2)};
\draw[very thick] plot [smooth,tension=.1]
coordinates{(-6,1)(-10,1)};
}
\]
\end{example}

\begin{example}\label{leftshiftexl} Let $e=3$, $\ul{s}=(0,2)$, and $\ul{\lambda}^t=((3^3,2),(3^3,2^2))$. We shift $\mathrm{Per}^3$ to the left as follows:
\[
\TikZ{[scale=.6]
\draw
(5,2)node[fill,circle,inner sep=3pt]{}
(5,1)node[fill,circle,inner sep=.5pt]{}
(4,2)node[fill,circle,inner sep=3pt]{}
(4,1)node[fill,circle,inner sep=.5pt]{}
(3,2)node[fill,circle,inner sep=3pt]{}
(3,1)node[fill,circle,inner sep=3pt]{}
(2,2)node[fill,circle,inner sep=.5pt]{}
(2,1)node[fill,circle,inner sep=3pt]{}
(1,2)node[fill,circle,inner sep=3pt]{}
(1,1)node[fill,circle,inner sep=3pt]{}
(0,2)node[fill,circle,inner sep=3pt]{}
(0,1)node[fill,circle,inner sep=.5pt]{}
(-1,2)node[fill,circle,inner sep=.5pt]{}
(-1,1)node[fill,circle,inner sep=3pt]{}
(-2,2)node[fill,circle,inner sep=.5pt]{}
(-2,1)node[fill,circle,inner sep=.5pt]{}
(-3,2)node[fill,circle,inner sep=3pt]{}
(-3,1)node[fill,circle,inner sep=.5pt]{}
(-4,2)node[fill,circle,inner sep=3pt]{}
(-4,1)node[fill,circle,inner sep=3pt]{}
(-5,2)node[fill,circle,inner sep=3pt]{}
(-5,1)node[fill,circle,inner sep=3pt]{}
(5,0)node[]{5}
(4,0)node[]{4}
(3,0)node[]{3}
(2,0)node[]{2}
(1,0)node[]{1}
(0,0)node[]{0}
(-1,0)node[]{-1}
(-2,0)node[]{-2}
(-3,0)node[]{-3}
(-4,0)node[]{-4}
(-5,0)node[]{-5}
;
\draw[very thick] plot [smooth,tension=.1]
coordinates{(5,2)(4,2)(3,1)};
\draw[very thick] plot [smooth,tension=.1]
coordinates{(3,2)(2,1)(1,1)};
\draw[very thick,Orange] plot [smooth,tension=.1]
coordinates{(1,2)(0,2)(-1,1)};
\draw[very thick] plot [smooth,tension=.1]
coordinates{(-3,2)(-4,1)(-5,1)};
\draw[very thick] plot [smooth,tension=.1]
coordinates{(-4,2)(-5,2)(-5.5,1.5)};
}
\qquad\rightsquigarrow\qquad
\TikZ{[scale=.6]
\draw
(5,2)node[fill,circle,inner sep=3pt]{}
(5,1)node[fill,circle,inner sep=.5pt]{}
(4,2)node[fill,circle,inner sep=3pt]{}
(4,1)node[fill,circle,inner sep=.5pt]{}
(3,2)node[fill,circle,inner sep=3pt]{}
(3,1)node[fill,circle,inner sep=3pt]{}
(2,2)node[fill,circle,inner sep=.5pt]{}
(2,1)node[fill,circle,inner sep=3pt]{}
(1,2)node[fill,circle,inner sep=.5pt]{}
(1,1)node[fill,circle,inner sep=3pt]{}
(0,2)node[fill,circle,inner sep=3pt]{}
(0,1)node[fill,circle,inner sep=.5pt]{}
(-1,2)node[fill,circle,inner sep=3pt]{}
(-1,1)node[fill,circle,inner sep=.5pt]{}
(-2,2)node[fill,circle,inner sep=.5pt]{}
(-2,1)node[fill,circle,inner sep=3pt]{}
(-3,2)node[fill,circle,inner sep=3pt]{}
(-3,1)node[fill,circle,inner sep=.5pt]{}
(-4,2)node[fill,circle,inner sep=3pt]{}
(-4,1)node[fill,circle,inner sep=3pt]{}
(-5,2)node[fill,circle,inner sep=3pt]{}
(-5,1)node[fill,circle,inner sep=3pt]{}
(5,0)node[]{5}
(4,0)node[]{4}
(3,0)node[]{3}
(2,0)node[]{2}
(1,0)node[]{1}
(0,0)node[]{0}
(-1,0)node[]{-1}
(-2,0)node[]{-2}
(-3,0)node[]{-3}
(-4,0)node[]{-4}
(-5,0)node[]{-5}
;
\draw[very thick] plot [smooth,tension=.1]
coordinates{(5,2)(4,2)(3,1)};
\draw[very thick] plot [smooth,tension=.1]
coordinates{(3,2)(2,1)(1,1)};
\draw[very thick,Orange] plot [smooth,tension=.1]
coordinates{(0,2)(-1,2)(-2,1)};
\draw[very thick] plot [smooth,tension=.1]
coordinates{(-3,2)(-4,1)(-5,1)};
\draw[very thick] plot [smooth,tension=.1]
coordinates{(-4,2)(-5,2)(-5.5,1.5)};
}
\]
\end{example}

\begin{theorem}\cite{GerberN} Let $\cA$ be a totally $e$-periodic abacus and let $\mathrm{Per}^k$, $k\geq 1$, be the $e$-periods of $\cA$. Suppose $\mathrm{Per}^k$ can be shifted to the right for some $k$ and let $\cA'$ be the abacus obtained from $\cA$ by shifting $\mathrm{Per}^k$ to the right. Then there is an edge from $\cA$ to $\cA'$ in the $\slinf$-crystal if and only if the right shift of $\mathrm{Per}^k$ is the $k$'th $e$-period of $\cA'$. All edges in the $\slinf$-crystal on the totally $e$-periodic abaci arise in this way. 
\end{theorem}

\begin{corollary}
Shifting the $e$-period $\mathrm{Per}^k$ to the left in a totally $e$-periodic abacus $\cA$ to obtain a new abacus $\cA'$ describes an edge in the $\slinf$-crystal if and only if the left shift of $\mathrm{Per}^k$ is the $k$'th $e$-period of $\cA'$.
\end{corollary}
\begin{example}\label{cusp abacus exl}
Let $e=3$ and consider the following totally $3$-periodic abacus with $\ul{\lambda}^t=((1^3),\emptyset)$ and $\ul{s}=(0,1)$, drawn with its $3$-periods:
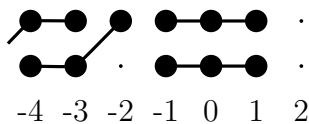
\[
\TikZ{[scale=.6]
\draw
(2,2)node[fill,circle,inner sep=.5pt]{}
(2,1)node[fill,circle,inner sep=.5pt]{}
(1,2)node[fill,circle,inner sep=3pt]{}
(1,1)node[fill,circle,inner sep=3pt]{}
(0,2)node[fill,circle,inner sep=3pt]{}
(0,1)node[fill,circle,inner sep=3pt]{}
(-1,2)node[fill,circle,inner sep=3pt]{}
(-1,1)node[fill,circle,inner sep=3pt]{}
(-2,2)node[fill,circle,inner sep=3pt]{}
(-2,1)node[fill,circle,inner sep=.5pt]{}
(-3,2)node[fill,circle,inner sep=3pt]{}
(-3,1)node[fill,circle,inner sep=3pt]{}
(-4,2)node[fill,circle,inner sep=3pt]{}
(-4,1)node[fill,circle,inner sep=3pt]{}
(2,0)node[]{2}
(1,0)node[]{1}
(0,0)node[]{0}
(-1,0)node[]{-1}
(-2,0)node[]{-2}
(-3,0)node[]{-3}
(-4,0)node[]{-4}
;
\draw[very thick] plot [smooth,tension=.1]
coordinates{(1,1)(-1,1)};
\draw[very thick] plot [smooth,tension=.1]
coordinates{(1,2)(-1,2)};
\draw[very thick] plot [smooth,tension=.1]
coordinates{(-2,2)(-3,1)(-4,1)};
\draw[very thick] plot [smooth,tension=.1]
coordinates{(-3,2)(-4,2)(-4.5,1.5)};
}
\]
Then shifting $\mathrm{Per}^1$ to the left \textit{does not} describe an edge in the $\slinf$-crystal, because the left shift of $\mathrm{Per}^1$ would not be the first $e$-period of the new abacus. Since no other $e$-period can shift to the left, the abacus is a source vertex in the $\slinf$-crystal. We conclude by Theorem \ref{source vx} that $L((3),\emptyset)$ is a finite-dimensional irreducible representation of the rational Cherednik algebra of $B_3$ with parameters $c=\frac{1}{3}$ and $d=-\frac{1}{6}$.
\end{example}

 Let $\Upsilon_k^-$ be the $\slinf$-crystal operator that shifts $\mathrm{Per}^k$ to the left if doing so yields the $k$'th $e$-period of the new abacus, and otherwise sends an abacus to $0$ \cite{GerberN}. Similarly define $\Upsilon_k^+$ for the right shift of $\mathrm{Per}^k$. The $\slinf$-crystal and the $\sle$-crystal operators commute \cite{Gerber1}. \footnote{Therefore, we may always compute the $\slinf$-crystal on an abacus by applying a sequence $\tilde{e}_{i_1},\ldots,\tilde{e}_{i_r}$ of $\sle$-crystal operators that remove boxes to arrive at a totally $e$-periodic abacus, then applying the $\slinf$-operators to the totally $e$-periodic abacus by the rule above, then applying the reverse sequence $\tilde{f}_{i_r}\ldots,\tilde{f}_{i_1}$ of $\sle$-crystal operators that add boxes.} 
 
 For an abacus $\cA$ associated to the charged bipartition $|\ul{\lambda}^t,\ul{s}\rangle\in\cF_{e,\ul{s}}$ and a partition $\sigma$, define $\tilde{a}_{\sigma}\cA$ to be the abacus obtained from $\cA$ by shifting $\mathrm{Per}^1$ (recursively) to the right $\sigma_1$ times, then shifting $\mathrm{Per}^2$ (recursively) to the right $\sigma^2$ times, and so on \cite{Gerber1}. That is, if $\sigma=(\sigma_1,\ldots,\sigma_r)$ then $\tilde{a}_\sigma=(\Upsilon^+_{r})^{\sigma_r}\ldots(\Upsilon^+_{1})^{\sigma_1}$. If $\cA$ is the abacus associated to $|\ul\la^t,\ul{s}\rangle\in\cF_{e,\ul{s}}^2$ and $|\ul{\mu}^t,\ul{s}\rangle\in\cF_{e,\ul{s}}^2$ is the charged bipartition associated to $\tilde{a}_{\sigma}\cA$, we will write $\ul\mu^t=\tilde{a}_{\sigma}(\ul\la^t)$. The crystal operator $\tilde{a}_{\sigma}$ is categorified by a functor $A_\sigma$ related to parabolic induction, see \cite{ShanVasserot},\cite{Losev}.
 
 On the bipartition $\ul{\lambda}^t$, the operator $\Upsilon_i^-$ removes a certain ``good vertical strip" of $e$ boxes of successive charged contents from the rim of $\ul{\lambda}^t$ or acts by $0$ (and thus on $\ul{\lambda}$, this results in the removal of a certain horizontal strip) \cite{Gerber2}. Similarly, $\Upsilon_k^+$ either acts by $0$ or adds a ``good vertical strip" to $\ul{\lambda}^t$ (respectively, results in the addition of a horizontal strip to $\ul{\lambda}$). The strip may be broken into two pieces, one piece in the first component and the other in the second component \cite{Gerber2}. 
 \begin{example} In Example \ref{leftshiftexl}, shifting $\mathrm{Per}^3$ to the left describes an edge in the $\slinf$-crystal. On the Young diagram of $\ul{\lambda^t}$, $\Upsilon_3^-$ removes the vertical $3$-strip consisting of the three shaded boxes:
 \begin{center}
\ydiagram[*(white)]{3,3,3,1}*[*(Orange)]{3,3,3,2}\quad,\quad\ydiagram[*(white)]{3,3,3,1,1}*[*(Orange)]{3,3,3,2,2}
\end{center}
On the Young diagram of $\ul{\lambda}$, this results in the removal of the horizontal $3$-strip consisting of the three shaded boxes:
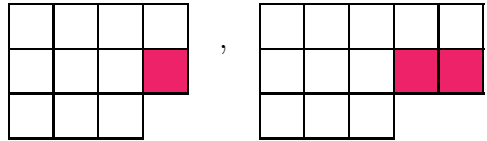
\begin{center}
\ydiagram[*(white)]{4,3,3}*[*(WildStrawberry)]{4,4,3}\quad,\quad\ydiagram[*(white)]{5,3,3}*[*(WildStrawberry)]{5,5,3}
\end{center}
 \end{example}
\begin{example}
Let $e=4$, $\ul{\lambda}=((3^2),(1^2))$, and $\ul{s}=(0,1)$. Then $|\ul{\lambda}^t,\ul{s}\rangle$ as an abacus is:
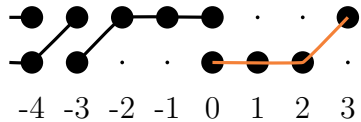
\[
\TikZ{[scale=.6]
\draw
(3,2)node[fill,circle,inner sep=3pt]{}
(3,1)node[fill,circle,inner sep=.5pt]{}
(2,2)node[fill,circle,inner sep=.5pt]{}
(2,1)node[fill,circle,inner sep=3pt]{}
(1,2)node[fill,circle,inner sep=.5pt]{}
(1,1)node[fill,circle,inner sep=3pt]{}
(0,2)node[fill,circle,inner sep=3pt]{}
(0,1)node[fill,circle,inner sep=3pt]{}
(-1,2)node[fill,circle,inner sep=3pt]{}
(-1,1)node[fill,circle,inner sep=.5pt]{}
(-2,2)node[fill,circle,inner sep=3pt]{}
(-2,1)node[fill,circle,inner sep=.5pt]{}
(-3,2)node[fill,circle,inner sep=3pt]{}
(-3,1)node[fill,circle,inner sep=3pt]{}
(-4,2)node[fill,circle,inner sep=3pt]{}
(-4,1)node[fill,circle,inner sep=3pt]{}
(3,0)node[]{3}
(2,0)node[]{2}
(1,0)node[]{1}
(0,0)node[]{0}
(-1,0)node[]{-1}
(-2,0)node[]{-2}
(-3,0)node[]{-3}
(-4,0)node[]{-4}
;
\draw[very thick,Orange] plot [smooth,tension=.1]
coordinates{(3,2)(2,1)(0,1)};
\draw[very thick] plot [smooth,tension=.1]
coordinates{(0,2)(-2,2)(-3,1)};
\draw[very thick] plot [smooth,tension=.1]
coordinates{(-3,2)(-4,1)(-4.5,1)};
\draw[very thick] plot [smooth,tension=.1]
coordinates{(-4,2)(-4.5,2)};
}
\]
In the $\slinf$-crystal we may shift $\mathrm{Per}^1$ to the left two times to obtain the empty bipartition. On the Young diagram of $\ul{\lambda}^t$, shifting $\mathrm{Per}^1$ to the left once removes the rightmost, darker-shaded vertical strip of boxes, then shifting $\mathrm{Per}^1$ to the left again removes the remaining, lighter-shaded vertical strip:
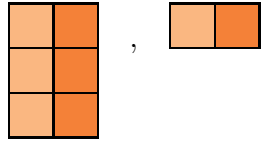
\begin{center}
\ydiagram[*(Apricot)]{1,1,1}*[*(Orange)]{2,2,2}\quad,\quad\ydiagram[*(Apricot)]{1}*[*(Orange)]{2}
\end{center}
On the Young diagram of $\ul{\lambda}$, shifting $\mathrm{Per}^1$ to the left once removes the bottom-most, darker-shaded horizontal strip of boxes, then shifting $\mathrm{Per}^1$ to the left again then removes the remaining, lighter-shaded horizontal strip:
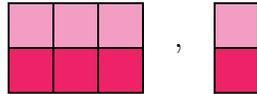
\begin{center}
\ydiagram[*(Lavender)]{3}*[*(WildStrawberry)]{3,3}\quad,\quad\ydiagram[*(Lavender)]{1}*[*(WildStrawberry)]{1,1}
\end{center}
We have $\ul{\lambda}^t=\tilde{a}_{(2)}(\emptyset,\emptyset)$.
\end{example}
The general rule for $\tilde{a}_{(m)}(\emptyset,\emptyset)$ is easily seen to be the following. Let $\ul{s}=(0,s)$. For $0<s<e$ we have $\tilde{a}_{(m)}(\emptyset,\emptyset)=((e-s)^m,(s^m))^t$. For $s\geq e$ we have $\tilde{a}_{(m)}(\emptyset,\emptyset)=(\emptyset,(e^m))^t$. For $s\leq 0$ we have $\tilde{a}_{(m)}(\emptyset,\emptyset)=((e^m),\emptyset)^t$.


\section{Rectangles parametrize finite-dimensional unitary irreducible representations}\label{section rectangle}
This section consists of the proof of Theorem \ref{unitary+fd}. The proof is broken into two steps. First, we use Lemma \ref{lemma ES} together with the conditions in \cite[Corollary 8.4]{Griffeth} to characterize the unitary finite-dimensional irreducible representations labeled by bipartitions with only one non-empty component. Second, we show that if $\lambda^1$ and $\lambda^2$ are both non-empty partitions, then $L(\lambda^1,\lambda^2)$ is never both unitary and finite-dimensional. 

\begin{proposition}\label{rectangle}
Fix $\ul{\lambda}\in\mathcal{P}^2$ such that $\ul{\lambda}=(\lambda,\emptyset)$. Then $L(\ul{\lambda})$ is a finite-dimensional and unitary $H_{e,\ul{s}}(B_n)$-representation for some Fock space parameter $(e,\ul{s})$ if and only if $\lambda$ is a rectangle. Conversely, given any rectangle $\lambda$, for each $e\geq 2$ there is a unique charge $\ul{s}$ (up to shift) such that $L(\lambda,\emptyset)$ is a finite-dimensional and unitary $H_{e,\ul{s}}(B_n)$-representation.
\end{proposition}

\begin{proof} By Lemma \ref{lemma ES} we must have $c_{\ul{\lambda}}=0$ if $\ul{\lambda}$ is unitary.
Assume that $c_{\ul{\lambda}}=0$. Then we have 
$$\frac{s}{e}=\frac{1}{n}\left( n-\frac{2}{e}\sum_{b\in\ul{\lambda}}\mathrm{ct}(b)\right)=1-\frac{2}{en}\sum_{b\in\ul{\lambda}}\mathrm{ct}(b)$$
and thus the parameters for the rational Cherednik algebra are 
$$c=\frac{1}{e},\qquad d=\frac{1}{2}-\frac{2}{en}\sum_{b\in\ul{\lambda}}\mathrm{ct}(b)$$

We suppose that $c_{\ul{\lambda}}=0$ and check when conditions (a)-(e) of \cite[Corollary 8.4]{Griffeth} can hold. For an integer $m$, consider the quantity
$$d+mc=\frac{1}{2}+\frac{1}{e}\left(m-\frac{2}{n}\sum_b\mathrm{ct}(b)\right)$$
under our assumption $c_{\ul{\lambda}}=0$. The inequality $d+mc\leq \frac{1}{2}$ holds if and only if $m\leq \frac{2}{n}\sum_b\mathrm{ct}(b)$, and $d+mc=\frac{1}{2}$ if and only if $m=\frac{2}{n}\sum_b\mathrm{ct}(b)$. 

Let $b_1$ be the box of largest content in $\lambda$ and let $b_2$ be the removable box of largest content in $\lambda$.
If $\lambda$ has $q$ columns then $\mathrm{ct}(b_1)=q-1$. The inequality $d+\mathrm{ct}(b_1)c\leq \frac{1}{2}$ holds if and only if $q-1\leq \frac{2}{n}\sum_b\mathrm{ct}(b)$. If $q=n$ then $\lambda=(n)$ and we have equality, and moreover, $b_2=b_1$ so case (d) holds and $L(n)$ is unitary. 
Suppose next that $\lambda$ is a rectangle with $q$ columns and $r$ rows where $r>1$. So $n=qr$ and $\frac{2}{n}\sum_b\mathrm{ct}(b)=\frac{2}{n}\left(\frac{n(q-r)}{2}\right)=q-r$. Since $r>1$ then we have $q-1>\frac{2}{n}\sum_b\mathrm{ct}(b)$, and thus $d+\mathrm{ct}(b_1)c>\frac{1}{2}$. Finally, if $\lambda$ is not a rectangle but its first row has $q$ boxes and $\lambda$ has $r$ rows, then clearly $q-r>\frac{2}{n}\sum_b\mathrm{ct}(b)$ so we also have $q-1>\frac{2}{n}\sum_b\mathrm{ct}(b)$. It follows that for arbitrary $\lambda\neq (n)$, cases (a) and (b) of \cite[Corollary 8.4]{Griffeth}  cannot occur. 

The remaining three cases (c),(d),(e) of \cite[Corollary 8.4]{Griffeth} all require the equation $d+mc=\frac{1}{2}$ to be satisfied for some integer $m$, so equivalently, $m=\frac{2}{n}\sum_{b}\mathrm{ct}(b)$ for some integer $m$. If $\lambda$ is a rectangle with $q$ columns and $r$ rows, $r>1$, then the solution is $m=q-r=\mathrm{ct}(b_2)$, so case (d) holds and $L(\lambda,\emptyset)$ is unitary. Otherwise, observe that cases (c),(d),(e) all require $m\geq \mathrm{ct}(b_2)$. Writing $\lambda=(q^r, \lambda_{r+1},\lambda_{r+2}\dots)$, $q>\lambda_{r+1}$, we have $\mathrm{ct}(b_2)=q-r$ which is $2$ times the average content of the boxes contained in the $q$ by $r$ rectangle comprising the first $r$ rows of $\lambda$. Adding boxes below this rectangle clearly lowers the average content of the boxes in the diagram, so we always have $m<\mathrm{ct}(b_2)$ if $\lambda$ is not a rectangle. Therefore cases (c),(d),(e) cannot hold if $\lambda$ is not a rectangle.

We conclude that given $e\geq 2$, a rectangle $\lambda$ with $r$ rows and $q$ columns, and $\ul{\lambda}=(\lambda,\emptyset)$, we have $c_{\ul{\lambda}}=0$ and $L(\ul{\lambda})$ unitary exactly when $s-e=r-q$. Now, $\hstar\otimes\ul{\lambda}$ is the irreducible $B_n$-representation labeled by $\ul{\mu}:=((q^{r-1},q-1),(1))$. When $s=e+r-q$, we compute that $\ul{\mu}=1$. Therefore $L(\lambda,\emptyset)$ is finite-dimensional and unitary exactly when $s=e-q+r$; and if $\lambda$ is not a rectangle, then $L(\lambda,\emptyset)$ is never finite-dimensional and unitary.
\end{proof}

\begin{remark}\label{switching components}
If $\ul{\lambda}=(\lambda,\emptyset)$ where $\lambda$ is a rectangle with $q$ columns and $r$ rows then $s=e-q+r$ is the smallest value of $s$ for which $L(\ul{\lambda})$ is finite-dimensional \cite{GerberN}.
\end{remark}
\begin{remark}\label{switch} Switching the components $\lambda^1$ and $\lambda^2$ is induced by an isomorphism of the underlying rational Cherednik algebras sending $d$ to $-d$ \cite[\S 2.3.4]{Losev}. Thus an analogous result to Proposition \ref{rectangle} holds for $L(\emptyset, \lambda)$: $L(\emptyset,\lambda)$ is unitary and finite-dimensional for Fock space parameter $(e,(s_1',s_2'))$ if and only if $L(\lambda,\emptyset)$ is unitary and finite-dimensional for Fock space parameter $(e,(s_1,s_2))$ where $s_2'-s_1'=e-(s_2-s_1)$. This justifies the conditions in Theorem \ref{unitary+fd} for when $L(\emptyset,\lambda)$ is unitary and finite-dimensional. 
\end{remark}

Next, we consider the case that both components of $\ul{\lambda}$ are non-empty. 
We will make an abacus argument. 
\begin{lemma}\label{>=e col} If $\ul{\lambda}=(\lambda^1,\lambda^2)$ with both $\lambda^1\neq \emptyset$ and $\lambda^2\neq \emptyset$ and the abacus $\cA(\ul{\lambda})$ associated to $|\ul{\lambda}^t,\ul{s}\rangle$ is totally $e$-periodic, then $\ul{\lambda}$ has at least $e$ nonzero columns. 
\end{lemma}
\begin{proof}Suppose $\cA(\ul{\lambda})$ is totally $e$-periodic and consider the maximal beta-number $\beta^j_1$ in $\cA(\ul{\lambda})$ (possibly it occurs twice, as $\beta^2_1$ and $\beta^1_1$). It must correspond to at least one nonzero column of $\ul{\lambda}$ since both components of $\ul{\lambda}$ are nonempty. 
Then either all of $\mathrm{Per}^1$ lies in the top row of $\cA(\ul{\lambda})$ and corresponds to $e$ nonzero columns of $\lambda^2$ of the same size, or for some $0\leq a\leq e-1$, the first $a$ beads of $\mathrm{Per}^1$ are $\beta^2_1,\dots \beta^2_a$ lying in the top row and these must correspond to $a$ nonzero columns of $\lambda^2$ since $\lambda^2\neq \emptyset$; while the remaining $e-a$ beads of $\mathrm{Per}^1$ are $\beta^1_{1},\dots, \beta^1_{e-a}$ lying in the bottom row and these must correspond to $e-a$ nonzero columns of $\lambda^1$ since $\lambda^1\neq \emptyset$. So $\ul{\lambda}$ has at least $e$ columns.
\end{proof}

\begin{lemma}\label{e col} If $\ul{\lambda}=(\lambda^1,\lambda^2)$ with both $\lambda^1\neq \emptyset$ and $\lambda^2\neq \emptyset$ and $\ul{\lambda}$ has exactly $e$ nonzero columns then $L(\ul{\lambda})$ is not finite-dimensional.
\end{lemma}
\begin{proof}Suppose $|\ul{\lambda}^t,\ul{s}\rangle$ is a source vertex in the $\sle$-crystal, then $\cA(\ul{\lambda})$ is totally $e$-periodic by  \cite[Theorem 5.9]{JaconLecouvey}.
As in the proof of Lemma \ref{>=e col}, the first $e$-period $\mathrm{Per}^1$ of  $\cA(\ul{\lambda})$ then corresponds to $e$ nonzero columns of $\ul{\lambda}$. By the assumption that $\ul{\lambda}$ has exactly $e$ columns, then the beads of $\mathrm{Per}^1$ exactly label the nonzero columns of $\ul{\lambda}$, and $\mathrm{Per}^1$ is as in the proof of Lemma \ref{>=e col} where the first $a$ beads of $\mathrm{Per}^1$ are in the top row of $\cA(\ul{\lambda})$ (corresponding to $a$ nonzero columns of $\lambda^2$), and the last $e-a$ beads of $\mathrm{Per}^1$ are in the bottom row of $\cA(\ul{\lambda})$, and $0<a<e$.  
But then the space in the bottom row to the left of $\beta_{e-a}^1$ and the space in the top row to the left of $\beta_a^2$ form a pair of spaces violating the condition in \cite[Theorem 7.13]{GerberN}, so $|\ul{\lambda},\ul{s}\rangle$ is not a source vertex in the $\slinf$-crystal, and in particular, $L(\ul{\lambda})$ is not finite-dimensional by Theorem \ref{source vx}. 
\end{proof}

\begin{proposition}\label{nonempty components} If $\ul{\lambda}=(\lambda^1,\lambda^2)$ with both $\lambda^1\neq \emptyset$ and $\lambda^2\neq \emptyset$ then $L(\ul{\lambda})$ is never both unitary and finite-dimensional.
\end{proposition}

\begin{proof}
We consider the cases (a)-(g) in \cite[Corollary 8.5]{Griffeth} one by one. First, we translate the notation $b_1,b_2,b_4,b_1',b_2',b_4'$ of \cite[Corollary 8.5]{Griffeth} into statements about numbers of rows and columns. We have $\ct(b_1)+1=\#\{\hbox{col}(\lambda^1)\}$, $\ct(b_1')+1=\#\{\hbox{col}(\lambda^2)\}$, $-\ct(b_4)+1=\#\{\hbox{row}(\lambda^1)\}$, $-\ct(b_4')+1=\#\{\hbox{row}(\lambda^2)\}$, $\ct(b_2)= \#\{\hbox{col}(\lambda^1)\}-\#\{\hbox{row}(\lambda^1)\hbox{ of size }\lambda_1^1\} $, and $\ct(b_2')= \#\{\hbox{col}(\lambda^2)\}-\#\{\hbox{row}(\lambda^2)\hbox{ of size }\lambda_1^2\} $. 

Case (a). The inequalities are:
\begin{align*}&-s\leq \#\{\hbox{columns}(\lambda^1)\}+\#\{\hbox{rows}(\lambda^2)\}-1\leq e-s,\\
&s-e\leq \#\{\hbox{columns}(\lambda^2)\}+\#\{\hbox{rows}(\lambda^1)\}-1\leq s.
\end{align*}
Since both $\#\{\hbox{rows}(\lambda^1)\}\geq 1$ and $\#\{\hbox{rows}(\lambda^2)\}\geq 1$ by the assumption $\lambda^1, \lambda^2\neq \emptyset$, when we add the inequalities we see that $\#\{\hbox{columns}(\ul{\lambda})\}\leq e$. 
Suppose $|\ul{\lambda}^t,\ul{s}\rangle$ is a source vertex in the $\sle$-crystal. Then $\cA(\ul{\lambda})$ is totally $e$-periodic by \cite[Theorem 5.9]{JaconLecouvey}. By Lemma \ref{>=e col} then $\ul{\lambda}$ has at least $e$ columns. So $\ul{\lambda}$ has exactly $e$ columns, but then by Lemma \ref{e col}, $L(\ul{\lambda})$ cannot be finite-dimensional.

Case (b). If $d+\ell c=\frac{1}{2}$ then $\ell=e-s$. We have: 
$$ \#\{\hbox{col}(\lambda^1)\}-\#\{\hbox{row}(\lambda^1)\hbox{ of size }\lambda_1^1\} + \#\{\hbox{row}(\lambda^2)\}\leq e-s.$$
Additionally, the second inequality required to hold is:
$$ \#\{\hbox{col}(\lambda^2)\}+\#\{\hbox{row}(\lambda^1)\}-1\leq s.$$
Adding inequalities, we have $$\#\{\hbox{col}(\ul{\lambda})\}\leq \#\{\hbox{col}(\ul{\lambda})\}+\left(\#\{\hbox{row}(\lambda^1)\}-\#\{\hbox{row}(\lambda^1)\hbox{ of size }\lambda_1^1\}\right) + \left(\#\{\hbox{row}(\lambda^2)\}-1\right)\leq e.$$ 
Now we conclude as in case (a) that $L(\ul{\lambda})$ is not finite-dimensional.

Case (c). If $-d+\ell c=\frac{1}{2}$ then $\ell=s$. The first inequality yields $$\#\{\hbox{col}(\lambda^2)\}-\#\{\hbox{row}(\lambda^2)\hbox{ of size }\lambda_1^2\} + \#\{\hbox{row}(\lambda^1)\}\leq s .$$
The second inequality yields $$\#\{\hbox{col}(\lambda^1)\}+\#\{\hbox{row}(\lambda^2)\}-1\leq e-s.$$ 
Now we add the inequalities and conclude as in case (b) that $L(\ul{\lambda})$ is not finite-dimensional.

Cases (d) and (e). These have $c$ replaced with $-c$ and $\ct(b)$ replaced with $-\ct (b)$ in all conditions, the latter being the same as considering $\lambda^t$ with the original conditions. But to deal with a Fock space with $-c$ instead of $c$ we just take charge $-\ul{s}$ and replace $\ul{\lambda}$ with $\ul{\lambda}^t$ \cite[Section 4.1.4]{Losev}. So cases (d) and (e) reduce to cases (b) and (c). 

Case (f). As in case (b), $d+\ell c=\frac{1}{2}$ implies $\ell=e-s$; and $ -d+mc=\frac{1}{2}$ implies $m=s$. We then have the inequalities:
$$\#\{\hbox{col}(\lambda^1)\}-\#\{\hbox{row}(\lambda^1)\hbox{ of size }\lambda_1^1\} + \#\{\hbox{row}(\lambda^2)\}\leq e-s,$$
$$\#\{\hbox{col}(\lambda^2)\}-\#\{\hbox{row}(\lambda^2)\hbox{ of size }\lambda_1^2\} + \#\{\hbox{row}(\lambda^1)\}\leq s.$$
We add the inequalities and conclude as in case (b).

Case(g). This case actually does not arise in the Fock space set-up: first, our assumption $c=\frac{1}{e}$, $d=-\frac{1}{2}+\frac{s}{e}$ determines $\ell=e-s$ and $m=s$ as in case (f). Then observe that $b_3$ is just ``$b_2$" for $(\lambda^1)^t$, etc, and the inequalities of case (g) if multiplied through by $-1$ become the inequalities of case (f), but for the transpose bipartition $\ul{\lambda}^t=((\lambda^1)^t, (\lambda^2)^t)$. But then we would have $\#\{\hbox{col}(\ul{\lambda}^t)\}\leq -e$, which is nonsense since $e>0$.
\end{proof}

Combining Propositions \ref{rectangle} and \ref{nonempty components}, and applying Remark \ref{switching components}, we conclude that Theorem \ref{unitary+fd} holds.


\section{Parabolic restriction of unitary representations in $\cO_{e,\ul{s}}(n)$}\label{Sec:Proof of Main Thm 2} 
In this section, we will prove Theorem \ref{Main2} and deduce Corollary \ref{Main3}. Theorem \ref{Main2} follows from two propositions. Proposition \ref{crystal ops unitaries} proves that the annihilation operators $\tilde{e}_i$ and $\Upsilon_k^-$ in the $\sle$- and $\slinf$-crystals on $\cF_{e,\ul{s}}^2$ preserve the combinatorial conditions of unitarity, and Proposition \ref{weak res unitaries semisimple}  proves that the functor $\EE_i$ sends unitary irreducible representations to irreducible representations or $0$. From the proof of Proposition \ref{crystal ops unitaries} we deduce Corollaries \ref{unitary supports} and \ref{slinf res}, yielding Corollary \ref{Main3} and the second statement of Theorem \ref{Main2}. Corollary \ref{unitary supports} classifies the unitary representations by their positions in the $\slinf$- and $\sle$-crystals. 

We also discuss some consequences of Corollary \ref{unitary supports} and Theorem \ref{Main2} in Section 
\ref{support rectangle}: many (asymptotically, one could say most) of the unitary representations $L(\ul\la)$ in $\cO_{e,\ul{s}}(n)$ are concentrated in a single component and are simply built from a rectangle $(q^r)$ labeling a finite-dimensional unitary representation as in Theorem \ref{unitary+fd} with a partition $\tau$ stacked underneath it such that $L(\tau)\in\cO_{\frac{1}{e}}(S_{|\tau|})$ is unitary. The supports of these types of unitary irreducible representations are then just read off from the supports of its two pieces, i.e. the support of such a unitary representation $L(\ul\la)$ is $B_{|\ul\la|}\mathfrak{h}^{W'}$ where $W'=B_{rq}\times W^{''}$ such that $W''$ yields the support of $L(\tau)$.

\subsection{The crystals and the unitarity conditions} As previously, let $(e,\ul{s})$ be a parameter for the level $2$ Fock space, so $e\geq 2$ is an integer and $\ul{s}\in\mathbb{Z}^2$. Let $\cO_{e,\ul{s}}(n)$ be the category $\cO$ of the rational Cherednik algebra of $B_n$ with parameters $c=\frac{1}{e}$ and $d=-\frac{1}{2}+\frac{s}{e}$ where $s=s_2-s_1$ and $\cO_{e,\ul{s}}=\bigoplus_{n\geq 0}\cO_{e,\ul{s}}(n)$. We can and will assume that $s_1=0$ and $s_2=s$. Define $\cU_{e,\ul{s}}^2\subset\cP^2$ to be the set of all bipartitions $\ul\la$ such that $L(\ul\la)\in\cO_{e,\ul{s}}$ is unitary. Later during discussion of the results we will start casually referring to elements of $\cU_{e,\ul{s}}^2$ as ``unitary bipartitions." During the proof we will be more formal. 
Recall that we identify $\ul{\lambda}=(\lambda^1,\lambda^2)\in\mathcal{P}^2$ with the charged bipartition $|\ul{\lambda}^t,\ul{s}\rangle\in\cF_{e,\ul{s}}^2$, where $\ul{\lambda}^t=((\lambda^1)^t,(\lambda^2)^t)$.

The combinatorial conditions for unitarity of $L(\ul\la)$ are given by \cite[Corollaries 8.4 and 8.5]{Griffeth} in terms of certain boxes $b_i$ in the Young diagram of $\ul\la$. In order to compute the crystals, we will also consider $\la^t$ with boxes $b_i^t$ corresponding to the boxes $b_i$ under the transpose map. 
 In Case (i) of our proofs we will consider $\ul\la=(\la,\emptyset)$ as in \cite[Corollary 8.4]{Griffeth}. In this case, particular boxes are marked in $\la$ and $\la^t$ as follows.
 \[
\la=\quad \begin{ytableau}
 *(white)& *(white)& *(white)& *(white)&b_1\\
 *(white)& *(white)& *(white)& *(white)&*(white)\\
 *(white)& *(white)& *(white)& *(white)&b_2\\
 *(white)& *(white)& b_3\\
 *(white)& *(white)& *(white)\\
 *(white)& *(white)& b_4\\
 b_5& *(white)\\
 \end{ytableau}
 \qquad\qquad
 \la^t=\quad\begin{ytableau}
*(white)& *(white)& *(white)& *(white)& *(white)& *(white)&b_5^t\\
*(white)& *(white)& *(white)& *(white)& *(white)& *(white)& *(white)\\
*(white)& *(white)& *(white)& b_3^t&*(white)& b_4^t\\
*(white)& *(white)& *(white)\\
b_1^t & *(white)& b_2^t\\
 \end{ytableau}
\]
The box $b_1$ is the box of largest content in $\la$, $b_2$ is the removable box of largest content, $b_4$ is the removable box of second-largest content, $b_3$ is the largest box in the same column as $b_4$ such that the vertical strip from $b_3$ to $b_4$ can be removed from $\la$, and $b_5$ is the box of smallest content in $\la$. Thus $b_1^t$ is the box of smallest content in $\la^t$, $b_5^t$ is the box of largest content in $\la^t$, and so on.

In Case (ii) of our proofs we will consider $\ul\la=(\la^1,\la^2)$ where $\la^1\neq \emptyset$ and $\la^2\neq \emptyset$, as in \cite[Corollary 8.5]{Griffeth}. In this case, particular boxes are marked in $\ul\la$ and $\ul\la^t$ as follows (the box $b_3$ is omitted because it will not be needed):
\[
\ul\la=(\la^1,\la^2)= \quad\begin{ytableau}
*(white)& *(white)& *(white)&b_1\\
*(white)& *(white)& *(white)&*(white)\\
*(white)& *(white)& *(white)&b_2\\
*(white)& *(white)& *(white)\\
b_4 & *(white)& *(white)\\
\end{ytableau}\quad,
\quad
\begin{ytableau}
*(white)& *(white)&b_1'\\
*(white)& *(white)& *(white)\\
*(white)& *(white)& *(white)\\
*(white)& *(white)& b_2'\\
*(white)& *(white)\\
b_5'\\
\end{ytableau}
\]
\[
\ul\la^t=((\la^1)^t,(\la^2)^t)= \quad\begin{ytableau}
*(white)& *(white)& *(white)&*(white)& b_4^t\\
*(white)& *(white)& *(white)& *(white)& *(white)\\
*(white)& *(white)& *(white)& *(white)& *(white)\\
b_1^t &*(white)& b_2^t\\
\end{ytableau}\quad,
\quad
\begin{ytableau}
*(white)& *(white)& *(white)& *(white)& *(white)& b_5^t\\
*(white)& *(white)& *(white)& *(white)& *(white)\\
{b_1'}^t&*(white)& *(white)& {b_2'}^t\\
\end{ytableau}
\]
In $\la^1$, $b_1$ is the box of largest content, $b_2$ is the removable box of largest content, $b_4$ is the box of smallest content; in $\la^2$, $b_1'$ is the box of largest content, $b_2'$ is the removable box of largest content, $b_4'$ is the box of smallest content. 

\begin{proposition}\label{crystal ops unitaries}
Let $\ul\la\in\cU_{e,\ul{s}}^2$, that is, suppose that $L(\ul{\lambda})\in\cO_{e,\ul{s}}$ is a unitary irreducible representation. Then
either $\tilde{e}_i(\ul\la^t)=0$ or $(\tilde{e}_i(\ul\la^t))^t\in\cU_{e,\ul{s}}^2$ for all $i\in\mathbb{Z}/e\mathbb{Z}$, and either 
$\Upsilon^-_k(\la^t)=0$ or $(\Upsilon^-_k(\la^t))^t\in\cU_{e,\ul{s}}^2$ for all $k\in\mathbb{N}$. That is, the operators in the $\sle$- and $\slinf$-crystals that remove boxes preserve the combinatorial conditions for unitarity.
\end{proposition}
\begin{proof}
We break the proof into two main cases: (i) when $\ul\la=(\la,\emptyset)$ for some $\la\in\cP$, and (ii) when $\ul\la=(\la^1,\la^2)$ with both $\la^1,\la^2\neq \emptyset$. The case that $\ul\la=(\emptyset,\la)$ follows from that for $\ul\la$ as in \cite[Corollary 8.4]{Griffeth}. In Case (i) we check the cases from \cite[Corollary 8.4]{Griffeth} for unitarity of $L(\la,\emptyset)$. In Case (ii) we check the cases from \cite[Corollary 8.5]{Griffeth} for unitarity of $L(\lambda^1,\lambda^2)$.\\


\noindent \ul{\em Case (i): $\ul\la=(\la,\emptyset)$ for some partition $\la$ of $n$.}\\
According to \cite[Corollary 8.4]{Griffeth}, $L((1^n),\emptyset)$ is unitary if and only if either $d\leq\frac{1}{2}$ or $d+\ell c=\frac{1}{2}$ for some $-n+1\leq \ell\leq -1$. If $d\leq \frac{1}{2}$ then this is independent of $n$, so the condition is preserved by removing a box. If $d+\ell c=\frac{1}{2}$ then with Fock space parameters we have $\ell=e-s$, yielding $n-1+e\geq s\geq e+1$. If $n\geq s-e+1$ then the condition is preserved by removing a box. If $n=s-e+1$ then $L(\ul\la,\emptyset)$ is a finite-dimensional unitary representation by Theorem \ref{unitary+fd}, so $\tilde{e}_i$ acts on it by $0$. We conclude that when $L((1^n),\emptyset)$ is unitary, $((1^n),\emptyset))^t=((n),\emptyset)$ is taken by $\tilde{e}_i$ to $((n-1),\emptyset)$ only when $L((1^{n-1}),\emptyset)$ is unitary, and otherwise $\tilde{e}_i$ sends it to $0$. No vertical strip of length bigger than $1$ can be removed from $(1^n)^t=(n)$, so $\Upsilon^-_k((n),\emptyset)=0$ for all $k\in\mathbb{N}$.

From now on, assume $\la\neq (1^n)$.


\ul{\em Case (i)(a)}. Writing \cite[Corollary 8.4, Case (a)]{Griffeth} in terms of Fock space parameters yields: 
	\begin{align*}
	&\ct(b_1)-\ct(b_5)+1\leq e,\\
	&\ct(b_1)\leq e-s.
	\end{align*}
	
\ul{Check that $\tilde{e}_i$ preserves unitarity}.
If $b_1=b_5$, that is if $\la=(1)$, then there is nothing to show, so we assume $b_1\neq b_5$. The operator $\tilde{e}_i$ is actually removing a box from $\la^t$, but let us examine the effect on $\la$. Obviously the inequalities are preserved if the box removed is not $b_1$ or $b_5$. Since $b_1$ is the box of largest content, $\ct(b_1)\geq 0$. If $b_1$ is removed then the box of largest content in $\la\setminus b_1$ has content $\ct(b_1)-1$, and all inequalities are preserved. Since $\ct(b_5)$ is the box of smallest content, $\ct(b_5)\leq 0$. If $b_5$ is removed then the box of smallest content in $\la\setminus b_5$ has content $\ct(b_5)+1$, and all inequalities are preserved. 

\ul{Check that $\Upsilon_k^-$ preserves unitarity}. The quantity $\ct(b_1)-\ct(b_5)+1$ is the maximum hooklength of a box in $\la$, equivalently in $\la^t$. Therefore the only way a vertical strip of length $e$ can be removed from $|(\la^t,\emptyset),\ul{s}\rangle$ is if $\la^t=(1^e)$, that is, if $\la=(e)$. Otherwise, all $\Upsilon^-_k$ act by $0$. Suppose $\la=(e)$. Then $\ct(b_1)=e-1$, so the second inequality is satisfied if and only if $s\leq 1$. If $s\leq 0$ then, considering $|(\la^t,\emptyset),\ul{s}\rangle$, we see that $\Upsilon^-_1(\ul\la^t)=\emptyset$, that is $\Upsilon^-_1$ acts by removing all the boxes of $\la$, while $\Upsilon^-_k(\ul\la^t)=0$ for $k>1$. If $s=1$, then by Theorem \ref{unitary+fd}, $L((e),\emptyset)$ is a unitary finite-dimensional representation, so all $\Upsilon_k^-$ act on $|\ul\la^t,\ul{s}\rangle$ by $0$.
	
	
\ul{\em Case (i)(b)}.
Writing \cite[Corollary 8.4, Case (b)]{Griffeth} in terms of Fock space parameters yields: 
	\begin{align*}
	&\ct(b_2)-\ct(b_5)+1\leq e,\\
	&\ct(b_1)\leq e-s.
	\end{align*}
	
\ul{Check that $\tilde{e}_i$ preserves unitarity}.
By the proof of Case (i)(a), if any box besides $b_2$ is removed then the boxes of the resulting partition satisfy the given inequalities. Now $b_2$ is the removable box of largest content, and if it is not equal to $b_1$ then the removable box of largest content in $\la\setminus b_2$ has content $\ct(b_2)+1$. This could only pose a problem if $\ct(b_2)-\ct(b_5)+1=e$ and $b_2$ is removed from $\la$. However, considering $|\la^t,\ul{s}\rangle$ in order to compute the crystal rule, we see that in this situation $\la^t$ has an addable box to the right of $b_5^t$ of content $-\ct(b_5)+1=-\ct(b_2)+e$. Then $b_2^t$ is not good removable, since it is cancelled either by this addable box or by the addable box in component $2$ of $\ul\la^t$.

\ul{Check that $\Upsilon_k^-$ preserves unitarity}. 
The condition that $\ct(b_2)-\ct(b_5)+1\leq e$ is the same condition that appears for type $A$ unitary representations at parameter $\frac{1}{e}$ as it is presented in \cite{BZGS}. It says that the bottom border of $\la$ has length at most $e$, as illustrated below:
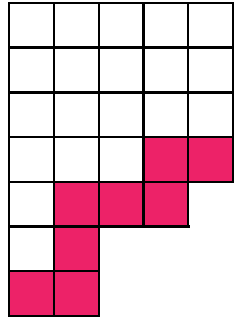
\begin{center}
\ydiagram[*(white)]{5,5,5,3,1,1}*[*(WildStrawberry)]{5,5,5,5,4,2,2}
\end{center}
In $\la^t$ we then have the following picture, with the shaded right border of $\la^t$ being required to have length at most $e$:
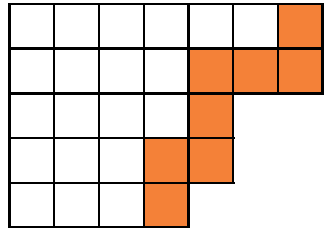
\begin{center}
\ydiagram[*(white)]{6,4,4,3,3}*[*(Orange)]{7,7,5,5,4}
\end{center}
The only way a vertical strip of length $e$ can be removed from $\la^t$, then, is if the shaded border strip is a vertical strip of length $e$, that is, if $\la^t$ is a rectangle with $e$ rows. Otherwise all $\Upsilon_k^-$ act by $0$. Now $\la^t$ is a rectangle with $e$ rows if and only if 
$\ct(b_2)-\ct(b_5)+1=e$ and $\ct(b_1)=e-1$. Imposing the condition $\ct(b_1)\leq e-s$, we see that it must hold that $s\leq 1$. 

Let us suppose, then, that $\la=(e^m)$ for some $m\geq 1$, so $\la^t=(m^e)$. The $\beta$-numbers yielding the abacus of $|(\la^t,\emptyset),\ul{s}\rangle$ are $\beta^1=\{m,m-1,\ldots,m-e+2,m-e+1,-e,-e-1,-e-2,\ldots\}$ and $\beta^2=\{s,s-1,s-2,\ldots\}$. The abacus is totally $e$-periodic with $\mathrm{Per}^1$ given by the first $e$ elements of $\beta^1$, as illustrated below for $s=0$, $e=8$, and $m=6$:

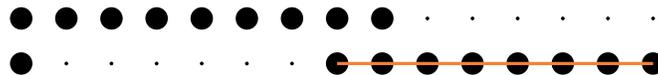
\[
\TikZ{[scale=.6]
\draw
(9,2)node[fill,circle,inner sep=.5pt]{}
(9,1)node[fill,circle,inner sep=3pt]{}
(8,2)node[fill,circle,inner sep=.5pt]{}
(8,1)node[fill,circle,inner sep=3pt]{}
(7,2)node[fill,circle,inner sep=.5pt]{}
(7,1)node[fill,circle,inner sep=3pt]{}
(6,2)node[fill,circle,inner sep=.5pt]{}
(6,1)node[fill,circle,inner sep=3pt]{}
(5,2)node[fill,circle,inner sep=.5pt]{}
(5,1)node[fill,circle,inner sep=3pt]{}
(4,2)node[fill,circle,inner sep=.5pt]{}
(4,1)node[fill,circle,inner sep=3pt]{}
(3,2)node[fill,circle,inner sep=3pt]{}
(3,1)node[fill,circle,inner sep=3pt]{}
(2,2)node[fill,circle,inner sep=3pt]{}
(2,1)node[fill,circle,inner sep=3pt]{}
(1,2)node[fill,circle,inner sep=3pt]{}
(1,1)node[fill,circle,inner sep=.5pt]{}
(0,2)node[fill,circle,inner sep=3pt]{}
(0,1)node[fill,circle,inner sep=.5pt]{}
(-1,2)node[fill,circle,inner sep=3pt]{}
(-1,1)node[fill,circle,inner sep=.5pt]{}
(-2,2)node[fill,circle,inner sep=3pt]{}
(-2,1)node[fill,circle,inner sep=.5pt]{}
(-3,2)node[fill,circle,inner sep=3pt]{}
(-3,1)node[fill,circle,inner sep=.5pt]{}
(-4,2)node[fill,circle,inner sep=3pt]{}
(-4,1)node[fill,circle,inner sep=.5pt]{}
(-5,2)node[fill,circle,inner sep=3pt]{}
(-5,1)node[fill,circle,inner sep=3pt]{}
;
\draw[very thick,Orange] plot [smooth,tension=.1]
coordinates{(9,1)(2,1)};
}
\]
As this is true for any $m\geq 1$, shifting $\mathrm{Per}^1$ to the left when $m>1$ produces the first $e$-period of $|((m-1)^e,\emptyset),\ul{s}\rangle$, and therefore describes the action of $\Upsilon^-_1$. In the case $m=1$, shifting $\mathrm{Per}^1$ produces the first $e$-period of the abacus of the empty bipartition if $s\leq 0$, and therefore describes the action of $\Upsilon_1^-$. If $s=1$ then shifting $\mathrm{Per}^1$ to the left when $m=1$ does not describe an edge in the $\slinf$-crystal, and $L((e),\emptyset)$ is a finite-dimensional unitary representation, as given by Theorem \ref{unitary+fd}. See Example \ref{cusp abacus exl} for a picture of this abacus when $e=3$. We conclude that if $s\leq 0$ then $|((m^e),\emptyset),\ul{s}\rangle=\tilde{a}_{(m)}|(\emptyset,\emptyset),\ul{s}\rangle$ for $m\geq 1$, while if $s=1$ then $|((m^e),\emptyset),(0,1)\rangle=\tilde{a}_{(m-1)}|((1^e),\emptyset),(0,1)\rangle$ for $m\geq 2$. And in particular, the operators $\Upsilon^-_k$ preserve the conditions of Case (i)(b).


\ul{\em Case (i)(c)}. 
Writing \cite[Corollary 8.4, Case (b)]{Griffeth} in terms of Fock space parameters yields: 
	\begin{align*}
	&\ct(b_2)<e-s\leq \ct(b_1),\\
	&-\ct(b_5)\leq s.
	\end{align*}

\ul{Check that $\tilde{e}_i$ preserves unitarity}. By assumption, $b_2\neq b_1$ so $b_1$ is not removable.
If any removable box besides $b_5,b_2$ is removed then obviously the inequalities are preserved. If $b_5$ is removed this increases the content of the box of smallest content, hence the inequalities are preserved. If $b_2$ is removed then there could only be a problem if $\ct(b_2)=e-s-1$. But then we are in Case (i)(d) so we refer to the argument in that case.

\ul{Check that $\Upsilon_k^-$ preserves unitarity}. If $\Upsilon^-_k$ does not act on $|(\la^t,\emptyset),\ul{s}\rangle$ by $0$ then $\la^t$ must contain a vertical strip of length $e$ in its right border. This implies $\ct(b_1)\geq e-1$. On the other hand, adding inequalities gives $\ct(b_2)-\ct(b_5)+1\leq e$. As in Case (i)(b), then $\la$ must be a rectangle with $e$ columns. Then $\la$ has $s+1$ rows and $s\geq 1$. If we consider $\mathrm{Per}^1$ of $|(s^e),\emptyset),(0,s)\rangle$, we see that its beads correspond to the $e$ rows of $(s^e)$. Then $\tilde{a}_{(1)}$ applied to $|(s^e),\emptyset),(0,s)\rangle$ adds a box to every row, yielding 
$\tilde{a}_{(1)}|(s^e),\emptyset),(0,s)\rangle=|(\la^t,\emptyset),(0,s)\rangle$. 
Thus $\Upsilon^-_1|(\la^t,\emptyset),(0,s)\rangle=|((e^s),\emptyset)^t,(0,s)\rangle$, which labels a finite-dimensional unitary representation by Theorem \ref{unitary+fd}. All $\Upsilon^-_k$ for $k>1$ act by $0$.


\ul{\em Case (i)(d)}. Writing \cite[Corollary 8.4, Case (d)]{Griffeth} in terms of Fock space parameters yields: 
	\[
	\ct(b_2)=e-s.
	\]
Additionally, if $\la$ is not a rectangle or if $b_4\neq b_5$, then either (a) for some $\ell$ such that $\ct(b_4)\leq \ell<\ct(b_3)$ it holds that $e=\ell-\ct(b_5)+1,$ or (b) $e\geq \ct(b_3)-\ct(b_5)+1$. Combining these last two conditions into one, we have that  if $\la$ is not a rectangle or if $b_4\neq b_5$, then 
	\[
	\ct(b_4)-\ct(b_5)+1\leq e.
	\]

Let $b_2^t$ be the box in $\la^t$ corresponding to $b_2$ under the transpose map. Then $\ct(b_2^t)=-\ct(b_2)=s-e$. By Lemma \ref{min cancellation pair}, $b_2^t$ is canceled by the unique addable box in component $2$ of $|\ul\la^t,\ul{s}\rangle$, so $b_2^t$ is not a good removable box of $|\ul\la^t,\ul{s}\rangle$. 

\ul{Check that $\tilde{e}_i$ and $\Upsilon_k^-$ preserve unitarity when $b_4=b_5$}. 
If $\lambda=(q^r)$ is a rectangle then $b_2=b_4=b_5$ and the unique removable box has content $s-e$. Then by Theorem \ref{unitary+fd}, $L(\ul\la)$ is finite-dimensional, and therefore $|\ul\la^t,\ul{s}\rangle$ is sent to $0$ by all $\tilde{e}_i$ and $\Upsilon_k^-$. If $b_4=b_5\neq b_2$ then $\la=(q^r,1^j)$ for some positive integers $q,r$ such that $r-q=s-e$ and for some $j\in\mathbb{N}$. For any $i\in\mathbb{Z}/e\mathbb{Z}$, we see that $b_5^t$ is always a good removable $i$-box of $|\ul\la,\ul{s}\rangle$ because it has the largest content of any removable box in $\la^t$, has larger content than any addable box in $\la^t$ of the same residue mod $e$, and the unique addable box in component $2$ cannot cancel it because that addable box cancels the box $b_2^t$. We deduce that $\tilde{e}_i$ acts by $0$ unless $i=-\ct(b_5)\mod e$ in which case $\tilde{e}_i$ removes $b_5^t$ from $|\ul{\la}^t,\ul{s}\rangle.$ This implies that $|\ul\la^t,\ul{s}\rangle=\tilde{f}_{r+j-1}\ldots\tilde{f}_{r+2}\tilde{f}_{r+1}\tilde{f}_{r}|(r^q),\ul{s}\rangle$ where the indices are taken mod $e$. For example, when $e=3$, $\ul{s}=(0,4)$, $r=5$, $q=4$, and $\la=(4^5,1^4)$ we have:
\begin{center}
$|\ul\la^t,\ul{s}\rangle=\tilde{f}_2\tilde{f}_1\tilde{f}_0\tilde{f}_2|((r^q),\emptyset),\ul{s}\rangle=$\quad\ydiagram[*(white)]{5,5,5,5}*[*(CornflowerBlue)]{6,5,5,5}*[*(CarnationPink)]{7,5,5,5}*[*(LimeGreen)]{8,5,5,5}*[*(CornflowerBlue)]{9,5,5,5}\quad,\quad  \large{$\emptyset$}
\end{center}
Since $|((r^q),\emptyset),\ul{s}\rangle$ is a source vertex in both the $\sle$- and $\slinf$-crystals, we have identified the position of $|\ul\la^t,\ul{s}\rangle $ in the two crystals: it has depth $0$ in the $\slinf$-crystal, and therefore $\Upsilon_k^-$ acts on it by $0$ for all $k\geq 1$.

\ul{Check that $\tilde{e}_i$ preserves unitarity when $b_4\neq b_5$.} Since $b_2^t$ is not a good removable box of $|\ul\la^t,\ul{s}\rangle$, applying $\tilde{e}_i$ preserves the condition $\ct(b_2)=e-s$. The condition $\ct(b_4)-\ct(b_5)+1\leq e$ continues to hold so long as any box is removed except $b_4$. In the case of $b_4$, removing it messes up the inequality if $\ct(b_4)-\ct(b_5)+1=e$. But then $-\ct(b_5)+1=-\ct(b_4)+e$, so $b_4^t$ is not a good removable box of $|\ul\la^t,\ul{s}\rangle$ (since $\la^2=\emp$). We illustrate this situation with a picture of such a $\la^t$ for $e=2$: 
\begin{ytableau}
*(white)&*(white)&*(white)&b_5^t&\none[\color{Thistle}\bullet]\\
*(white)&*(white)&*(white)&*(Thistle)b_4^t\\
*(white)&b_2^t\\
\end{ytableau}.\\

\ul{Check that $\Upsilon_k^-$  preserves unitarity when $b_4\neq b_5$.} Since $\ct(b_4)-\ct(b_5)+1\leq e$ and $b_4$ is the second-largest removable box of $\la$, there are at most two vertical $e$-strips that can be removed from $\la^t$. We illustrate this scenario for $e=2$:
\[
\begin{ytableau}
*(white)&*(white)&*(white)&*(Orange)b_5^t\\
*(white)&*(white)&*(white)&*(Orange)b_4^t\\
*(white)&*(white)\\
*(white)&*(white)\\
*(white)&*(Orange)\\
*(white)&*(Orange)b_2^t\\
\end{ytableau}\quad,\quad \large{\emptyset}
\qquad\leftrightarrow\qquad
\TikZ{[scale=.6]
\draw
(8,0)node[fill,circle,inner sep=.5pt]{}
(8,-1)node[fill,circle,inner sep=.5pt]{}
(7,0)node[fill,circle,inner sep=.5pt]{}
(7,-1)node[fill,circle,inner sep=3pt]{}
(6,0)node[fill,circle,inner sep=.5pt]{}
(6,-1)node[fill,circle,inner sep=3pt]{}
(5,0)node[fill,circle,inner sep=.5pt]{}
(5,-1)node[fill,circle,inner sep=.5pt]{}
(4,0)node[fill,circle,inner sep=.5pt]{}
(4,-1)node[fill,circle,inner sep=.5pt]{}
(3,0)node[fill,circle,inner sep=.5pt]{}
(3,-1)node[fill,circle,inner sep=3pt]{}
(2,0)node[fill,circle,inner sep=.5pt]{}
(2,-1)node[fill,circle,inner sep=3pt]{}
(1,0)node[fill,circle,inner sep=3pt]{}
(1,-1)node[fill,circle,inner sep=3pt]{}
(0,0)node[fill,circle,inner sep=3pt]{}
(0,-1)node[fill,circle,inner sep=3pt]{}
(-1,0)node[fill,circle,inner sep=3pt]{}
(-1,-1)node[fill,circle,inner sep=.5pt]{}
(-2,0)node[fill,circle,inner sep=3pt]{}
(-2,-1)node[fill,circle,inner sep=.5pt]{}
(-3,0)node[fill,circle,inner sep=3pt]{}
(-3,-1)node[fill,circle,inner sep=3pt]{}
(1,-2)node[]{{\small$s$}}
(-1,-2)node[]{{\small$s-e$}}
;
\draw[very thick,Orange] plot [smooth,tension=.1]
coordinates{(7,-1)(6,-1)};
\draw[very thick,Orange] plot [smooth,tension=.1]
coordinates{(1,-1)(0,-1)};
}
\]
However, even if $b_2^t$ sits at the bottom of a vertical strip of length $e$, removing that strip never yields an edge in the $\slinf$-crystal. This is because $\ct(b_2^t)=s-e$, so in the abacus associated to $|(\la^t,\emptyset),(0,s)\rangle$ the corresponding chain of $e$ beads on the bottom row of the abacus lies entirely underneath a chain of $e$ beads in the top row of the abacus. Such a chain of $e$ beads can never be shifted to the left by an operator $\Upsilon_k^-$ in the $\slinf$-crystal \cite{GerberN}. It follows that $\Upsilon_k^-$ acts on $|\ul\la^t,\ul{s}\rangle$ by $0$ unless $k=1$ and $\ct(b_4)-\ct(b_5)+1=e$.

Let us examine the scenario that $\ct(b_4)-\ct(b_5)+1=e$ and $b_5^t$, $b_4^t$ are the top and bottom boxes of a vertical strip in the border of $\la^t$. Then $b_2$ and $b_4$ are the only removable boxes of $\la$, so $\la=(q^r,e^m)$ for some positive integers $q,r,m$ satisfying $q>e$ and $r-q=s-e$. We then have $|(\la^t,\emptyset),(0,s)\rangle=\tilde{a}_{(m)}|((r^q),\emptyset),(0,s)\rangle$. In particular, $\Upsilon^-_1|\ul\la^t,\ul{s}\rangle=|((r+m-1)^e,r^{q-e}),(0,s)\rangle$ and corresponds to removing the bottom row of $e$ boxes from $\la$ yielding $(q^r,e^{m-1})$. The partition $(q^r,e^{m-1})$ again falls under Case (i)(d).


\ul{\em Case (i)(e)}.
Writing \cite[Corollary 8.4, Case (e)]{Griffeth} in terms of Fock space parameters yields: 
	\begin{align*}
	&\ct(b_2)+1\leq e-s\leq \ct(b_1)-1,\\
	&\ct(b_2)-\ct(b_5)+1\leq e\leq e-s-\ct(b_5).
	\end{align*}

\ul{Check that $\tilde{e}_i$ preserves unitarity}. If $b_1$ is removed and $\ct(b_1)=e-s+1$ then the resulting partition belongs to Case (i)(b). If $b_5$ is removed and $s=-\ct(b_5)$ then the resulting partition belongs to Case (i)(c). If $b_2$ is removed, $\ct(b_2)+1=e-s$, and $s\neq-\ct(b_5)$, then the resulting partition belongs to Case (i)(d).
Finally, as in Case (i)(b), if $\ct(b_2)-\ct(b_5)+1= e$ then $b_2^t$ is not a good removable box of $|\ul\la^t,\ul{s}\rangle$. 

\ul{Check that $\Upsilon_k^-$ preserves unitarity}.
As in Case (i)(b), the $\slinf$-crystal operators $\Upsilon_k^-$ act by $0$ unless $\la$ is a rectangle of width $e$. This follows from the requirement that $\ct(b_2)-\ct(b_5)+1\leq e$. So suppose that $\la=(e^m)$. Then $\ct(b_1)-1=e-1$, so $s\geq 1$. From the requirement $\ct(b_2)<e-s$ we deduce that $m>s$.  Reasoning as in Part (i)(d), we compute that $|(\la^t,\emptyset),(0,s)\rangle=\tilde{a}_{m-s}|((s^e),\emptyset),(0,s)\rangle$. Thus successively removing the bottom row of $\la$ describes the effect of repeated iterations of $\Upsilon_1^-$ on $\la$ until the rectangle $(e^s)$ is reached. In particular Case (i)(e) is preserved by applying $\Upsilon_1^-$ if $m>s+1$, while for $m=s+1$, applying $\Upsilon_1^-$ lands the resulting partition in Case (i)(d). All other $\Upsilon^-_k$ act by $0$ since $|(\la^t,\emptyset),(0,s)\rangle=\tilde{a}_{m-s}|((s^e),\emptyset),(0,s)\rangle$.

We illustrate the situation $|(\la^t,\emptyset),(0,s)\rangle=\tilde{a}_{m-s}|((s^e),\emptyset),(0,s)\rangle$ in the last paragraph when $e=3$, $s=2$, and $m=6$. The $e$-period $\mathrm{Per}^1$ can repeatedly shift to the left yielding the iterated action of $\Upsilon_1^-$ until it ``docks" under the top row of beads. Thus it can shift to the left under the action of $\Upsilon_1^-$ exactly four times. Once $\mathrm{Per}^1$ is completely covered by beads in the top row, shifting it to the left no longer describes an edge in the $\slinf$-crystal, thus $(\Upsilon^-_1)^5$ acts by $0$. As is consistent with Theorem \ref{unitary+fd}.

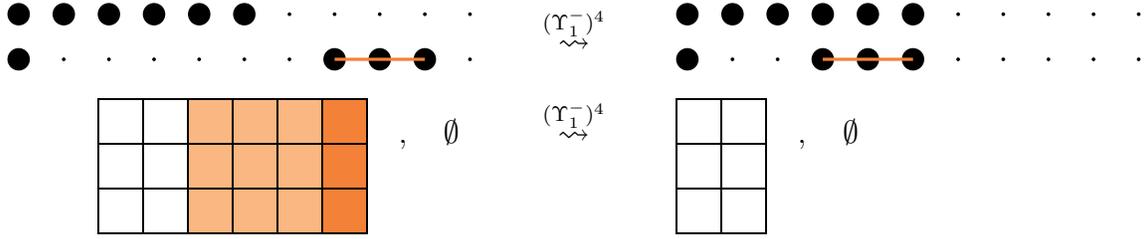
\begin{align*}
\TikZ{[scale=.6]
\draw
(10,2)node[fill,circle,inner sep=.5pt]{}
(10,1)node[fill,circle,inner sep=.5pt]{}
(9,2)node[fill,circle,inner sep=.5pt]{}
(9,1)node[fill,circle,inner sep=3pt]{}
(8,2)node[fill,circle,inner sep=.5pt]{}
(8,1)node[fill,circle,inner sep=3pt]{}
(7,2)node[fill,circle,inner sep=.5pt]{}
(7,1)node[fill,circle,inner sep=3pt]{}
(6,2)node[fill,circle,inner sep=.5pt]{}
(6,1)node[fill,circle,inner sep=.5pt]{}
(5,2)node[fill,circle,inner sep=3pt]{}
(5,1)node[fill,circle,inner sep=.5pt]{}
(4,2)node[fill,circle,inner sep=3pt]{}
(4,1)node[fill,circle,inner sep=.5pt]{}
(3,2)node[fill,circle,inner sep=3pt]{}
(3,1)node[fill,circle,inner sep=.5pt]{}
(2,2)node[fill,circle,inner sep=3pt]{}
(2,1)node[fill,circle,inner sep=.5pt]{}
(1,2)node[fill,circle,inner sep=3pt]{}
(1,1)node[fill,circle,inner sep=.5pt]{}
(0,2)node[fill,circle,inner sep=3pt]{}
(0,1)node[fill,circle,inner sep=3pt]{}
;
\draw[very thick,Orange] plot [smooth,tension=.1]
coordinates{(9,1)(7,1)};
}
\qquad
&\stackrel{(\Upsilon_1^-)^4}{\rightsquigarrow}
\qquad
\TikZ{[scale=.6]
\draw
(10,2)node[fill,circle,inner sep=.5pt]{}
(10,1)node[fill,circle,inner sep=.5pt]{}
(9,2)node[fill,circle,inner sep=.5pt]{}
(5,1)node[fill,circle,inner sep=3pt]{}
(8,2)node[fill,circle,inner sep=.5pt]{}
(4,1)node[fill,circle,inner sep=3pt]{}
(7,2)node[fill,circle,inner sep=.5pt]{}
(3,1)node[fill,circle,inner sep=3pt]{}
(6,2)node[fill,circle,inner sep=.5pt]{}
(6,1)node[fill,circle,inner sep=.5pt]{}
(5,2)node[fill,circle,inner sep=3pt]{}
(9,1)node[fill,circle,inner sep=.5pt]{}
(4,2)node[fill,circle,inner sep=3pt]{}
(8,1)node[fill,circle,inner sep=.5pt]{}
(3,2)node[fill,circle,inner sep=3pt]{}
(7,1)node[fill,circle,inner sep=.5pt]{}
(2,2)node[fill,circle,inner sep=3pt]{}
(2,1)node[fill,circle,inner sep=.5pt]{}
(1,2)node[fill,circle,inner sep=3pt]{}
(1,1)node[fill,circle,inner sep=.5pt]{}
(0,2)node[fill,circle,inner sep=3pt]{}
(0,1)node[fill,circle,inner sep=3pt]{}
;
\draw[very thick,Orange] plot [smooth,tension=.1]
coordinates{(5,1)(3,1)};
}\\
\ydiagram[*(white)]{2,2,2}*[*(Apricot)]{5,5,5}*[*(Orange)]{6,6,6}\quad,\quad\large{\emptyset}
\qquad
&\stackrel{(\Upsilon_1^-)^4}{\rightsquigarrow}
\qquad
\ydiagram[*(white)]{2,2,2}\quad,\quad\large{\emptyset}
\end{align*}\\


\noindent \ul{\em Case (ii): $\ul\la=(\la^1,\la^2)$ with both $\la^1,\la^2\neq \emptyset$.}\\
 It is sufficient to check the conditions for unitarity in cases (a), (b), and (f) of \cite[Corollary 8.5]{Griffeth}. The other cases (c), (d), (e), and (g) are obtained from these by switching parameter $c$ with $-c$ or parameter $d$ with $-d$, which have the effect of sending $\ul\la $ to $\ul\la^t$ or to $(\la^2,\la^1)$, respectively.


\ul{\em Case (ii)(a)}. Writing \cite[Corollary 8.5, Case (a)]{Griffeth} in terms of Fock space parameters yields: 
	\begin{align*}
	&-s\leq \ct(b_1)-\ct(b_4')+1\leq e-s,\\
	&s-e\leq \ct(b_1')-\ct(b_4)+1\leq s.
	\end{align*}
	
\ul{Check that $\tilde{e}_i$ preserves unitarity}.
Note that $\ct(b_1),\ct(b_1')\geq 0$ and $\ct(b_4),\ct(b_4')\leq 0$ always and thus $\ct(b_1)-\ct(b_4')+1> 0$ and $\ct(b_1')-\ct(b_4)+1>0$ always. The two right-hand inequalities then imply $0<s<e$, making the two left-hand inequalities irrelevant. Removing any box of $\ul\la$, the quantities $\ct(b_1)-\ct(b_4')+1$ and $\ct(b_1')-\ct(b_4)+1$ either stay the same or decrease by $1$. It follows that $(\tilde{e}_i(\ul\la^t))^t$ also satisfies the conditions of Case (a) so long as both of its components are non-empty. 

Next, suppose $\la^2=(1)$ and $\tilde{e}_i$ removes the unique box of $(\la^2)^t$. Then $b_1'=b_4'$ and $\ct(b_1')=0$. Then $\ct(b_1)<e-s$ and $\ct(b_1)-\ct(b_4)+1<e$, so \cite[Corollary 8.4, Case (a)]{Griffeth} holds for $(\tilde{e}_i(\ul\la)^t)^t$. On the other hand, if $\la^1=(1)$ then switching $\la^1$ with $\la^2$ and setting $s'=e-s$ (i.e. replacing $d$ by $-d$), we make the same argument.

\ul{Check that $\Upsilon_k^-$ preserves unitarity}. Adding the two inequalities, we have $(\ct(b_1)-\ct(b_4)+1)+(\ct(b_1')-\ct(b_4')+1)\leq e$. The quantity $\ct(b_1)-\ct(b_4)+1$ is the maximum hooklength of a box in $\la^1$ and $\ct(b_1')-\ct(b_4')+1$ is the maximum hooklength of a box in $\la^2$. It follows that the abacus of $|\ul\la^t,\ul{s}\rangle$ has an $e$-period that can shift to the left only if $(\la^1)^t$ and $(\la^2)^t$ are columns, in which case $b_1=b_2$, $b_1'=b_2'$, and we are in Case (ii)(f). Otherwise $\Upsilon_k^-$ acts by $0$ for all $k\in\mathbb{N}$. We now refer to the proof of Case (ii)(f) in the situation that  $(\la^1)^t$ and $(\la^2)^t$ are columns and $|\ul\la|=e$.

	
\ul{\em Case (ii)(b)}. Writing \cite[Corollary 8.5, Case (b)]{Griffeth} in terms of Fock space parameters yields: 
	\begin{align*}
	&\ct(b_2)-\ct(b_4')+1\leq e-s\leq \ct (b_1)-\ct(b_4')+1,\\
	&\ct(b_1')-\ct(b_4)+1\leq s.
	\end{align*}

\ul{Check that $\tilde{e}_i$ preserves unitarity}.
	First, we assume that both components of $\tilde{e}_i(\ul\la^t)$ are nonempty.
Removing any box of $\ul\la$ preserves the bottom inequality, removing any box different from $b_2$ preserves the top left inequality, and removing any box different from $b_1$ or $b_4'$ preserves the top right inequality. If the top right inequality is not preserved then $e-s=\ct(b_1)-\ct(b_4')+1$ and $\tilde{e}_i$ removes $b_1^t$ or $b_4'^t$ from $\ul\la^t$, but then Case (ii)(a) holds. The top left inequality must be preserved unless $\ct(b_2)-\ct(b_4')+1=e-s$ and $b_2\neq b_1$. Writing $\cct(b_4'^t)+1=-\ct(b_4')+s+1=-\ct(b_2)+e=\cct(b_2^t)+e$, we see that by Lemma \ref{min cancellation pair}, $b_2^t$ is not a good removable box. Therefore $\tilde{e}_i$ preserves the top left inequality.

Suppose $\la^2=(1)$ and that $\tilde{e}_i$ removes the unique box of $\la^2=(\la^2)^t$. Then $\ct(b_4')=\ct(b_1')=0$, so $\ct(b_2)+1\leq e-s\leq \ct(b_1)+1$ and $-\ct(b_4)+1\leq s$. Then $(\la^1,\emptyset)=(\tilde{e}_i(\ul\la^t))^t$ satisfies \cite[Corollary 8.4, Case (c)]{Griffeth} if $e-s\leq \ct(b_1)$. If $e-s=\ct(b_1)+1$, then $\ct(b_1)<e-s$ and $\ct(b_2)-\ct(b_4)+1<e$, so $(\la^1,\emptyset)$ then satisfies \cite[Corollary 8.4, Case (b)]{Griffeth}.

Suppose $\la^1=(1)$ and that $\tilde{e}_i$ removes the unique box of $\la^1=(\la^1)^t$. Then $\ct(b_4)=\ct(b_1)=\ct(b_2)=0$, so $-\ct(b_4')+1= e-s$ and $\ct(b_1')+1\leq s$. Since $L(\emptyset,\lambda^2)$ for parameters $(c,d)$ is unitary if and only if $L(\lambda^2,\emptyset)$ is unitary for parameters $(c,-d)$, we see that $L(\emptyset,\lambda^2)$ is unitary by \cite[Corollary 8.4, Case (a)]{Griffeth}.

\ul{Check that $\Upsilon_k^-$ preserves unitarity}. The same analysis holds as in Case (ii)(a). The abacus $|\ul\la^t,\ul{s}\rangle$ has an $e$-period that can shift to the left only if $|\ul\la|=e$ and $(\la^1)^t$, $(\la^2)^t$ are both columns. This situation falls under Case (ii)(f). Otherwise, $\Upsilon_k^-$ acts by $0$ for all $k\in\mathbb{N}$.


\ul{\em Case (ii)(f)}. Writing \cite[Corollary 8.5, Case (f)]{Griffeth} in terms of Fock space parameters yields: 
	\begin{align*}
	&\ct(b_2)-\ct(b_4')+1\leq e-s\leq \ct(b_1)-\ct(b_4')+1,\\
	&\ct(b_2')-\ct(b_4)+1\leq s\leq \ct(b_1')-\ct(b_4)+1.
	\end{align*}

\ul{Check that $\tilde{e}_i$ preserves unitarity}. 		First, we assume that both components of $\tilde{e}_i(\ul\la^t)$ are nonempty.
The two right-hand inequalities are preserved by removing any boxes except when $e-s= \ct(b_1)-\ct(b_4')+1$ or $s= \ct(b_1')-\ct(b_4)+1.$ However, in either of those cases then we are in Case (ii)(b) or \cite[Corollary 8.5, Case (c)]{Griffeth}, which follows from Case (ii)(b) as discussed at the beginning of the proof. We have the same analysis of the top left inequality as in Case (ii)(b). A similar argument applies to the bottom left inequality: it must be preserved by removing any box unless $b_2'\neq b_1'$ and $\cct(b_2'^t)=-\ct(b_2')+s=-\ct(b_4)+1=\cct(b_4^t)+1$. But in this case, $b_2'^t$ is not a good removable box of $|\ul\la^t,\ul{s}\rangle$ by Lemma \ref{min cancellation pair}.
Therefore $\tilde{e}_i$ always preserves the bottom left inequality.

The case where $\la^j=(1)$ and that box is removed is the same as the argument in Case (ii)(b), except that \cite[Corollary 8.4, Case (b)]{Griffeth} now arises when $\la^1=(1)$ and $s<\ct(b_1')+1$.

\ul{Check that $\Upsilon_k^-$ preserves unitarity}. Adding the two leftmost inequalities for Case (ii)(f), the number of  boxes in the right border strip of $\ul\la^t$ as shaded in the diagram below, is at most $e$:
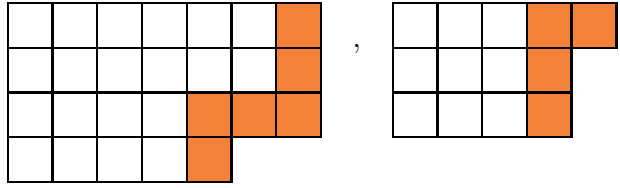
\begin{center}
\ydiagram[*(white)]{6,6,4,4}*[*(Orange)]{7,7,7,5}\quad,\quad\ydiagram[*(white)]{3,3,3}*[*(Orange)]{5,4,4}
\end{center}
On the other hand, if $\Upsilon_k^-$ does not act by $0$ on $\ul\la^t$ then it removes some vertical strip of $e$ boxes from the right border strip of $\ul\la^t$, a subset of the shaded region. This can only happen, therefore, if the shaded region is a vertical strip of $e$ boxes, as illustrated below for $e=7$:
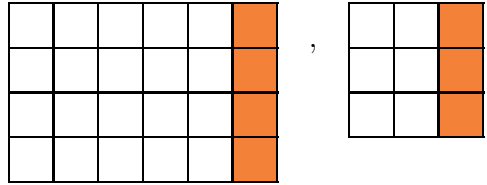
\begin{center}
\ydiagram[*(white)]{5,5,5,5}*[*(Orange)]{6,6,6,6}\quad,\quad\ydiagram[*(white)]{2,2,2}*[*(Orange)]{3,3,3}
\end{center} 
We then have that $\la^1$ and $\la^2$ are rectangles, and $\cct(b_4'^t)+1=-\ct(b_4')+s+1=-\ct(b_2)+e=\cct(b_2^t)+e$ and $\cct(b_4^t)+1=-\ct(b_4)+1=-\ct(b_2')+s=\cct(b_2^t)$. By Lemma \ref{min cancellation pair}, the two removable boxes of $\ul\la^t$ are not good removable so $|\ul\la^t,\ul{s}\rangle$ is totally $e$-periodic. The $\beta$-numbers are $\beta^1=\{-\ct(b_4)+1,-\ct(b_4),\ldots,-\ct(b_2)+2,-\ct(b_2)+1,-\ct(b_1)-1,-\ct(b_1)-2,-\ct(b_1)-3,\ldots\}$ and $\beta^2=\{-\ct(b_4')+s+1,-\ct(b_4')+s,\ldots,-\ct(b_2')+s+2,-\ct(b_2')+s+1,-\ct(b_1')+s-1,-\ct(b_1')+s-2,-\ct(b_1')+s-3,\ldots\}$.
In a picture:
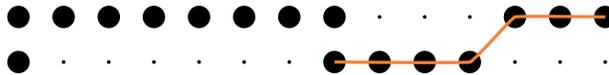
\[
\TikZ{[scale=.6]
\draw
(8,2)node[fill,circle,inner sep=3pt]{}
(8,1)node[fill,circle,inner sep=.5pt]{}
(7,2)node[fill,circle,inner sep=3pt]{}
(7,1)node[fill,circle,inner sep=.5pt]{}
(6,2)node[fill,circle,inner sep=3pt]{}
(6,1)node[fill,circle,inner sep=.5pt]{}
(5,2)node[fill,circle,inner sep=.5pt]{}
(5,1)node[fill,circle,inner sep=3pt]{}
(4,2)node[fill,circle,inner sep=.5pt]{}
(4,1)node[fill,circle,inner sep=3pt]{}
(3,2)node[fill,circle,inner sep=.5pt]{}
(3,1)node[fill,circle,inner sep=3pt]{}
(2,2)node[fill,circle,inner sep=3pt]{}
(2,1)node[fill,circle,inner sep=3pt]{}
(1,2)node[fill,circle,inner sep=3pt]{}
(1,1)node[fill,circle,inner sep=.5pt]{}
(0,2)node[fill,circle,inner sep=3pt]{}
(0,1)node[fill,circle,inner sep=.5pt]{}
(-1,2)node[fill,circle,inner sep=3pt]{}
(-1,1)node[fill,circle,inner sep=.5pt]{}
(-2,2)node[fill,circle,inner sep=3pt]{}
(-2,1)node[fill,circle,inner sep=.5pt]{}
(-3,2)node[fill,circle,inner sep=3pt]{}
(-3,1)node[fill,circle,inner sep=.5pt]{}
(-4,2)node[fill,circle,inner sep=3pt]{}
(-4,1)node[fill,circle,inner sep=.5pt]{}
(-5,2)node[fill,circle,inner sep=3pt]{}
(-5,1)node[fill,circle,inner sep=3pt]{}
;
\draw[very thick,Orange] plot [smooth,tension=.1]
coordinates{(8,2)(6,2)(5,1)(2,1)};
}
\]
The only $e$-period which can shift to the left is $\mathrm{Per}^1$ which consists of the $\beta$-numbers recording the non-zero parts of $\ul\la^t$. Shifting $\mathrm{Per}^1$ to the left yields the action of $\Upsilon^-_1$ and  $\Upsilon^-_k$ acts by $0$ for all $k\geq 2$. Meanwhile $\Upsilon^-_1(\ul\la^t)$ replaces every part $(\la^j)^t_i$ of $\ul\la^t$ with $(\la^j)_i^t-1$, i.e. it removes the shaded vertical strip which is the whole right border strip of $\ul\la^t$. If both components of $\Upsilon^-_1(\ul\la^t)$ are nonzero then it satisfies Case (ii)(f), so then $(\Upsilon^-_1(\ul\la^t))^t\in\cU_{e,\ul{s}}^2$. If $(\la^1)^t$ or $(\la^2)^t$ is a column but the other is not, then $\Upsilon^-_1(\ul\la^t)=(\la^t,\emptyset)$ or $(\emptyset,\la^t)$ labels a finite-dimensional unitary representation as analyzed in Section \ref{section rectangle}. If both $(\la^1)^t$ and $(\la^2)^t$ are columns then $\Upsilon^-_1(\ul\la^t)=(\emptyset,\emptyset)$, which labels the trivial representation of $\mathbb{C}=H_{c,d}(B_0)$.
\end{proof}

\subsection{Classification of the supports of the unitary representations} 

For $\ul\mu\in\cP^2$, set $\tilde{a}_{(0)}\ul\mu=\ul\mu$. For $I=(i_r,i_{r-1},\ldots,i_1)$ a sequence of elements $i_j\in\mathbb{Z}/e\mathbb{Z}$ set $\tilde{f}_I=\tilde{f}_{i_r}\tilde{f}_{i_{r-1}}\ldots \tilde{f}_{i_1}$.
\begin{corollary}\label{unitary supports}
Let $\ul\la\in\cU^2_{e,\ul{s}}$. Then $|\ul\la^t,\ul{s}\rangle\in\cF_{e,\ul{s}}^2$ has one of the following six descriptions as a vertex in the $\sle$- and $\slinf$-crystals on $\cF_{e,\ul{s}}^2$:\footnote{To be pedantic but completely precise, $\ul\la^t$ satisfies exactly one of the six possibilities unless $\ul\la^t$ has depth $0$ in both the $\sle$- and $\slinf$-crystals, in which case $\ul\la^t$ satisfies either (1) and (2), or (3) and (4), or (5) and (6), taking $I$ to be the empty sequence.}
\begin{enumerate}
\item $\ul\la^t=\tilde{f}_I(\emptyset,\emptyset)$ for an appropriate sequence $I$ of elements in $\mathbb{Z}/e\mathbb{Z}$,\footnote{We do not address which sequences $I$ may occur.}
\item  $\ul\la^t=\tilde{a}_{(m)}(\emptyset,\emptyset)$ for some $m\in\mathbb{N}\cup\{0\}$,
\item  $\ul\la^t=\tilde{f}_I((r^q),\emptyset)$ for an appropriate sequence $I$ of elements in $\mathbb{Z}/e\mathbb{Z}$  and some $r,q\in\mathbb{N}$ satisfying $r-q=s-e$,
\item $\ul\la^t=\tilde{a}_{(m)}((r^q),\emptyset)$  for some $m\in\mathbb{N}\cup\{0\}$ and some $r,q\in\mathbb{N}$ satisfying $r-q=s-e$,
\item $\ul\la^t=\tilde{f}_I(\emptyset,(r^q))$ for an appropriate sequence $I$ of elements in $\mathbb{Z}/e\mathbb{Z}$  and some $r,q\in\mathbb{N}$ satisfying $r-q=-s$,
\item  $\ul\la^t=\tilde{a}_{(m)}(\emptyset,(r^q))$  for some $m\in\mathbb{N}\cup\{0\}$ and some $r,q\in\mathbb{N}$ satisfying $r-q=-s$.
\end{enumerate}
Moreover, for a fixed $n$, all these situations are realized by some $\ul\la\in\cU_{e,\ul{s}}^2(n)$ when it is numerically possible, i.e.: for any $n$, (1) occurs; for any $n$ divisible by $e$, (2) occurs; whenever $rq\leq n$ and $r-q=s-e$, (3) occurs, and if $n-rq$ is also divisible by $e$ then (4) occurs; whenever $rq\leq n$ and $r-q=-s$, (5) occurs, and if $n-rq$ is also divisible by $e$ then (6) occurs. In particular, Corollary \ref{Main3} holds.
\end{corollary}
\begin{proof}
The statements (1) through (6) and examples of existence of all cases follow from the case-by-case examinations in the proof of Proposition \ref{crystal ops unitaries} together with Theorem \ref{unitary+fd}. Theorem \ref{Main3} is the same statement in the language of supports by \cite{Losev}.
\end{proof}

Observe that if $\ul\la\in\cU_{e,\ul{s}}^2$ then $\ul\la^t$ cannot have non-zero depth in both the $\sle$- and $\slinf$-crystals. Moreover, we see that the position $\sigma$ of $\ul\la^t$ in the $\slinf$-crystal is always given by $\sigma=(m)$ for $m\in\mathbb{N}$ if $\ul\la\in\cU_{e,\ul{s}}^2$. 

\begin{corollary}\label{slinf res}
Suppose $L(\ul\mu)\in\cO_{e,\ul{s}}(n+me)$ is a unitary irreducible representation with support equal to $B_{n+me}\mathfrak{h}^{W'}$ for $W'=B_n\times S_e^m$. Then (i) $\Ores^{B_{n+em}}_{B_{n+e(m-1)}\times S_e}L(\ul\mu)$ is irreducible and equal to $L((\tilde{a}_{(m-1)}\ul\la^t)^t)\otimes L(e)$, and (ii)  $\Ores^{B_{n+em}}_{B_n\times S_e^{\times m}}L(\ul\mu)$ is irreducible and equal to $L(\ul\la)\otimes L(e)^{\otimes m}$, where $\ul\la=((q^r),\emp)$  with $r-q=s-e$ or $\ul\la=(\emp,(q^r))$ with $r-q=-s$.
\end{corollary}

\begin{proof}
By Corollary \ref{unitary supports}, $\ul\mu=(\tilde{a}_{(m)}\ul\la^t)^t$ for some $\ul\la$ as stipulated above, that is, $\ul\mu^t$ has depth $m$ in the $\slinf$ crystal, its position in its connected component of that crystal is given by the partition $(m)$, and the source vertex of its connected component is some $\ul\la$ as above. 
The statement now follows from \cite[Proposition 4.3]{LosevCat} and adjointness.
\end{proof}

The supports of unitary representations are now classified as a set, but we can further unpack the implications of the proof of Proposition \ref{crystal ops unitaries} by identifying which supports occur for which kinds of bipartitions in $\cU_{e,\ul{s}}$, and for which values of $s$. We see a very similar pattern to what happens for the partitions labeling unitary representations in type $A$, as we explain next.

Let us recall the condition for $L(\tau)\in\cO_{\frac{1}{e}}(S_n)$ to be unitary. Take $b_4$ to be the removable box of $\tau$ of largest content, and take $b_5$ to be the box of $\tau$ of smallest content. Then $L(\tau)$ is unitary if and only if $b_4-b_5+1\leq e$. That is to say, the bottom border of $\tau$ as shaded in the diagram below must contain at most $e$ boxes:
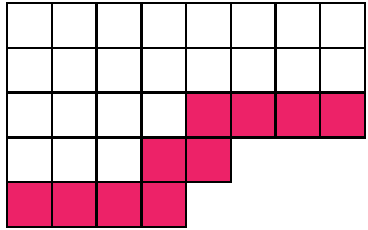
\begin{center}
\ydiagram[*(white)]{8,8,4,3}*[*(WildStrawberry)]{8,8,8,5,4}
\end{center}
 This follows from the classification given in \cite{EtingofStoica} and is a rephrasing of the $e$-abacus condition given in \cite[Section 4]{BZGS}. In particular, a partition $\tau$ such that $L(\tau)$ is unitary has at most $e$ columns, and such $\tau$ has $e$ columns if and only if $\tau=(e^m)$ for some $m\in\mathbb{N}$.
 Set $\cO_e=\bigoplus_{n\geq 0}\cO_{\frac{1}{e}}(S_n)$ and set $\cU_e=\{\tau\in\cP\mid L(\tau)\in\cO_e \hbox{ is unitary}\}$.
 
If $\tau$ has less than $e$ columns then $L(\tau)$ always has cuspidal support $(L(\triv),\{1\})$, i.e. $L(\tau)$ has full support \cite{Wilcox}. On the other hand, if $\tau=(e^m)$ then the cuspidal support of $L(\tau)$ is $(L(e)^{\otimes m}, S_e^{\times m})$ \cite{Wilcox}. The $\slinf$- and $\sle$-crystals on the set of all partitions $\cP$, which is equal to the level $1$ Fock space, are defined similarly as on $\cF^2_{e,\ul{s}}$. (However, in level $1$ there is a shift by $1$ in the $\slinf$-crystal as the finite-dimensional representation $L(e)$ has depth $1$, not depth $0$, in the $\slinf$-crystal: $(e)^t=\tilde{a}_{(1)}\emptyset$.) The two crystals describe the supports of $L(\tau)\in\cO_c(S_n)$. 
If $\tau\in\cU_e$ and $\tau\neq (e^m)$ then $\tau^t=\tilde{f}_I\emptyset$ for an appropriate sequence $I$ of $n$ elements of $\mathbb{Z}/e\mathbb{Z}$. If $\tau=(e^m)\in\cU_e$ then $\tau^t=\tilde{a}_{(m)}\emptyset$.

\subsubsection{Unitary bipartitions with type $B$ support}\label{support rectangle} We are going to look at which types of bipartitions in $\cU_{e,\ul{s}}$ can have type $B$ support or a mix of type $B$ and type $A$ supports. 
As we saw in the proof of Proposition \ref{crystal ops unitaries}, taking $I=(r+j,r+j-1,\ldots,r)$ for any $j\geq 1$ and applying $\tilde{f}_I$ to a rectangle labeling a finite-dimensional unitary representation produces an instance of situation (5) above. But more generally, \cite[Corollary 8.4, Case (d)]{Griffeth}, i.e. Case (i)(d) in the proof of Proposition \ref{crystal ops unitaries}, deals with partitions $\la$ such that $\la^t=(r^q)+\tau^t$ for some partition $\tau$ with at most $q-1$ columns. It says that $L(\la,\emptyset)\in\cO_{e,\ul{s}}$ is unitary if and only if $r-q=s-e$ and $L(\tau)\in\cO_{\frac{1}{e}}(S_{|\tau|})$ is unitary. This occurs for any parameter $(e,\ul{s})$.

In \cite[Corollary 8.4, Case (d)]{Griffeth}, $\la$ consists of a rectangle $(q^r)$ labeling a finite-dimensional unitary representation with a type $A$ unitary partition $\tau$ stacked underneath. The cuspidal support of $L(\la,\emptyset)$ is then read directly off of $\la$ from the cuspidal supports of its two pieces: the cuspidal support consists of $L((q^r),\emptyset)\in\cO_{e,\ul{s}}(rq)$ times the cuspidal support of $L_{\frac{1}{e}}(\tau)\in\cO_\frac{1}{e}(S_{|\tau|})$.  If $\la\neq (e^m)$ then we have $(\la^t,\emptyset)=\tilde{f}_I((r^q),\emptyset)$ for an appropriate sequence $I$ of elements of $\mathbb{Z}/e\mathbb{Z}$. It is easy to see that each such sequence $I$ is the shift of a sequence producing $\tau$ in the level $1$ crystal on $\cP$; that is, the sequence $I-r:=(i_{|\tau|}-r,\ldots i_1-r)$ satisfies $\tilde{f}_{I-r}\emptyset=\tau$. 


\begin{example} Let $\ul{s}=(0,3)$ (so $s=3$) and $e=5$. We look at the types of supports occurring in Case (i)(d) of the proof of Proposition \ref{crystal ops unitaries}, i.e. in \cite[Corollary 8.4, Case (d)]{Griffeth}, taking the ``cuspidal rectangle" $(q^r)=(6^4)$ and stacking a ``type $A$ unitary partition" $\tau$ beneath it.  We have $\ct(b_2)=-2=r-q=s-e$.

(a) Let $\ul\la=(\la,\emptyset)=((6^4),\emptyset)$. Then $L(\ul\la)\in\cO_{e,\ul{s}}(24)$ is a finite-dimensional unitary representation.
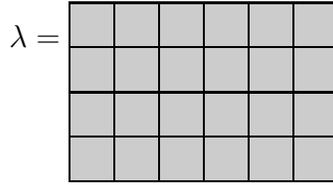
\[
\la=\ydiagram[*(LightGray)]{6,6,6,6}
\]
The partition $\tau$ below the rectangle is the empty partition. The cuspidal support of $L(\ul\la)$ is $(L(\ul\la),B_{24})$.
The depth of $\ul\la^t=((4^6),\emp)$ in the $\slinf$- and $\widehat{\mathfrak{sl}}_5$-crystals is $0$.

(b) Let $\ul\la=(\la,\emptyset)=((6^4,1^2),\emptyset)$ so that $b_4=b_5$. We break $\la$ into two pieces, the ``cuspidal rectangle" on top of the ``type A unitary partition"  $(1^2)$: 
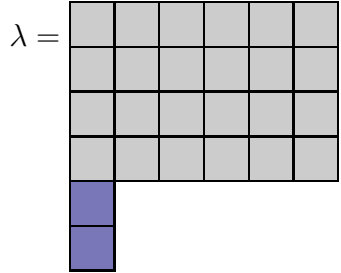
\[
\la=\ydiagram[*(LightGray)]{6,6,6,6}*[*(Periwinkle)]{6,6,6,6,1,1}
\]
The cuspidal support of $L(\ul\la)$ is $(L((6^4),\emptyset)\otimes L(\triv)^{\otimes 2},B_{24}\times S_1^2)\simeq (L((6^4),\emptyset),B_{24})$ and $\ul\la^t=\tilde{f}_5\tilde{f}_4((4^6),\emptyset)$.

(c) Let $\ul\la=(\la,\emptyset)=((6^4,3^3,2,1),\emptyset)$. Then $\ct(b_4)-\ct(b_5)+1=5=e$. We break $\la$ into two pieces, the ``cuspidal rectangle" on top of the ``type A unitary":
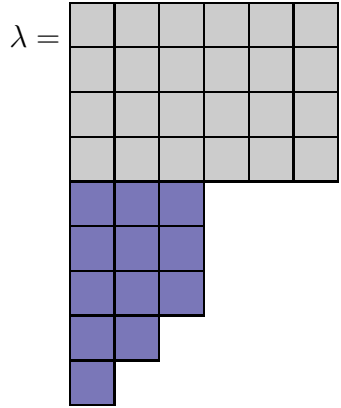
\[
\la=\ydiagram[*(LightGray)]{6,6,6,6}*[*(Periwinkle)]{6,6,6,6,3,3,3,2,1}
\]
We then read off the cuspidal support of $L(\ul\la)$ from the two pieces: it is equal to $(L((6^4),\emptyset), B_{24})$ times the cuspidal support of $L(3^3,2,1)\in\cO_{\frac{1}{5}}(S_{12}).$ Since $\tau$ is not a rectangle of width $5$, we conclude that the cuspidal support of $L(\ul\la)$ is $(L((6^4),\emptyset)\otimes L(\triv)^{\otimes 12},B_{24}\otimes S_1^{ 12})\simeq (L((6^4),\emp),B_{24})$ and that $\ul\la^t=\tilde{f}_I((4^6),\emp)$ for any sequence $I=(i_{12},\ldots, i_1)$ such that $ \tilde{f}_{i_{12}-4}\ldots\tilde{f}_{i_1-4}\emptyset=\tau$ in the $\widehat{\mathfrak{sl}}_5$-crystal on $\cP$.

(d) Let $\ul\la=((6^4,5^3),\emptyset)$. Then $\ct(b_4)-\ct(b_5)+1=5=e$.  We break $\la$ into two pieces, the ``cuspidal rectangle" on top of the ``type A unitary":
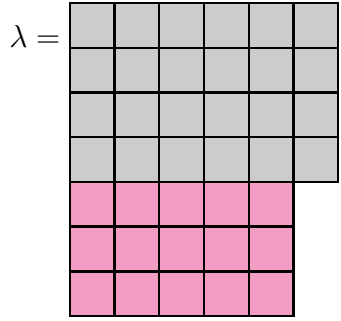
\[
\la=\ydiagram[*(LightGray)]{6,6,6,6}*[*(Lavender)]{6,6,6,6,5,5,5}
\]
Since $\tau$ is a rectangle of width $5=e$ and height $3$, $\tau$ has depth $3$ in the $\slinf$-crystal on $\cP$ and is obtained as $\tilde{a}_{(3)}(\emptyset)=\tau^t$. In the $\slinf$-crystal on $\cF_{e,\ul{s}}^2$, we have $\ul\la^t=\tilde{a}_{(3)}((4^6),\emptyset)$. The cuspidal support of $L(\ul\la)$ is $(L((6^4),\emptyset)\otimes L(5)^{\otimes 3},B_{24}\times S_5^3)$.
\end{example}

In Case (i)(e) a similar phenomenon can occur, but with the width of $\tau$ allowed to be equal to the width $q$ of the ``cuspidal rectangle" $(q^r)$ on top of it for $q\leq e$.

\begin{example} Let $e=7$ and $\ul{s}=(0,6)$. Then for any rectangle $(q^r)$ with $q-r=1$, $L((q^r),\emp)\in\cO_{e,\ul{s}}(qr)$ is finite-dimensional and unitary.  Let $\ul\la=((3^6,1),\emptyset)$. The cuspidal support of $L(\ul\la)$ is $(L((3^2),\emp)\otimes L(\triv)^{13},B_{6}\times S_1^{13})\simeq (L((3^2),\emp),B_6)$:
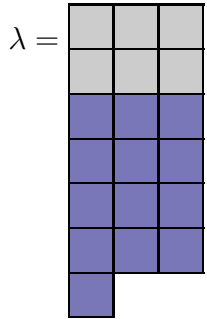
\[
\la=\ydiagram[*(LightGray)]{3,3}*[*(Periwinkle)]{3,3,3,3,3,3,1}
\]
\end{example}

Next, let us discuss when $L(\ul\la)$, $\ul\la=(\la^1,\la^2)\in\cU_{e,\ul{s}}$ with both $\la^1,\la^2$ nonempty, can have part of its support of type $B$. 
First, consider the case that the depth of $\ul\la^t$ in the $\sle$-crystal is nonzero (so that its depth in the $\slinf$-crystal is $0$). We saw in the proof that if $\ul\la=(\la^1,\la^2)$ with $\la^j=(1)$ for $j=1 $ or $j=2$ and both $\la^1,\la^2\neq \emptyset$, then the cases that arise when removing the unique box from $\la^j$ are \cite[Corollary 8.4, Cases (a) and (b)]{Griffeth}. Assume $\la^2=(1)$ is removed by some $\tilde{e}_i$, the other case $\la^1=(1)$ being similar. Recall that \cite[Corollary 8.4, Cases (a) and (b)]{Griffeth} require $\ct(b_1)\leq e-s$. It is impossible for any $(\la,\emptyset)$ belonging to these cases to arise as $(\tilde{f}_I((r^q),\emptyset))^t$ with $r-q=s-e$, except when $\la=(e-s+1)$ is a row of length at most $e$ with $0<s\leq e$. It follows that if both components $\la^1$ and $\la^2$ of $\ul\la$ are nonempty, $\ul\la\in\cU_{e,\ul{s}}$, and the depth of $\ul\la^t$ in the $\slinf$-crystal is $0$, then one of the following holds:
\begin{itemize}
\item $\ul\la^t=\tilde{f}_I(\emptyset,\emptyset)$ for an appropriate sequence $I$, 
\item $0<s\leq e$ and $\ul\la^t=\tilde{f}_I((e-s+1)^t,\emptyset)$ for an appropriate sequence $I$.
\item  $0\leq s<e$ and $\ul\la^t=\tilde{f}_I(\emptyset, (s+1)^t)$ for an appropriate sequence $I$.
\end{itemize}
The latter two situations do occur.
\begin{example}
Take $e=5$ and $\ul{s}=(0,3)$. Let $\ul\la=((3),\emptyset).$ Then $L(\ul\la)$ is unitary and finite-dimensional. Applying $\tilde{f}_3\tilde{f}_4\tilde{f}_0\tilde{f}_1$ to $\ul\la^t$ we get $\tilde{f}_I\ul\la^t=((2^3),(1))$ for $I=(3,4,0,1)$. Then $((2^3),(1))^t=((3^2),(1))\in\cU_{e,\ul{s}}$ by \cite[Corollary 8.5, Case (b)]{Griffeth}.
\end{example}

We have already looked at the case of nonzero depth in the $\slinf$-crystal in detail in Case (ii)(f) of the proof of Proposition \ref{crystal ops unitaries}. Let $r-q=s-e$. Let $\ul\la^t=\tilde{a}_{(m)}((r^q),\emptyset)$. Then $\ul\la$ will have its second component $\la^2$ non-empty if and only if $\mathrm{Per}^1$ of the abacus $|((r^q),\emptyset),\ul{s}\rangle$ contains beads from the top row of the abacus. This is the case if and only if $e>q$.

\begin{example}
Let $\ul\la=((4^{14}),(1^2))$. If $e=5$ and $\ul{s}=(0,13)$ then $\ul\la^t=\tilde{a}_{(2)}|((12^4),\emptyset),\ul{s}\rangle$. The cuspidal support of $L(\ul\la)\in\cO_{e,\ul{s}}$ is $(L((12^4),\emptyset)\otimes L(5)^{\otimes 2},B_{48}\times S_5^{\times 2})$. 
\end{example}

\subsection{Semisimplicity and restriction} We now return to establishing Theorem \ref{Main2}.
\begin{proposition}\label{weak res unitaries semisimple}
Let $L(\ul\la)\in\cO_{e,\ul{s}}(n)$ be a unitary irreducible representation. Then either $\Ores^n_{n-1}L(\ul\la)$ is semisimple or it is $0$.
\end{proposition}
\begin{proof}
As in Proposition \ref{crystal ops unitaries}, we break the proof into two main cases: (i) when $\ul\la=(\la,\emptyset)$ for some $\la\in\cP(n)$, and (ii) when $\ul\la=(\la^1,\la^2)$ with both $\la^1,\la^2\neq \emptyset$. The case that $\ul\la=(\emptyset,\la)$ follows from that for $\ul\la$ as in \cite[Corollary 8.4]{Griffeth}. In Case (i) we check the cases from \cite[Corollary 8.4]{Griffeth} for unitarity of $L(\la,\emptyset)$. In Case (ii) we check the cases from \cite[Corollary 8.5]{Griffeth} for unitarity of $L(\lambda^1,\lambda^2)$.\\


\noindent \ul{\em Case (i): $\ul\la=(\la,\emptyset)$ for some partition $\la$ of $n$.} \\
If $\la=(1^n)$ then $\la$ has only one removable box, so Lemma \ref{easy i-res} says that $\Ores^n_{n-1}L(\ul\la)$ is irreducible or $0$. From now on, assume $\la\neq (1^n)$. 

\ul{\em Cases (i)(a),(b),(c),(e)}. Since $\ct(b_2)\leq \ct(b_1)$, in Cases (a), (b), and (e) it holds that $\ct(b_2)-\ct(b_5)+1\leq e$. In Case (c) we have $\ct(b_5)\geq -s$ and $\ct(b_2)<e-s$, yielding $\ct(b_2)-\ct(b_5)+1\leq e$ as well in this case. The contents of removable boxes in $\la$ belong to the interval of integers $[-\ct(b_5),\ct(b_2)]$ which contains $\ct(b_2)-\ct(b_5)+1$ distinct integers.  
Thus no $i\in\mathbb{Z}/e\mathbb{Z}$ appears more than once as $\ct(b)\mod e$ for a removable box $b\in\la$. By Lemma \ref{easy i-res} then $\Ores^n_{n-1}L(\ul\la)$ is semisimple or $0$. 

\ul{\em Case(i)(d)}. In case (d) it is required that $\ct(b_2)=e-s$. If $\la$ is a rectangle or $b_4=b_5$ then this is the only condition. We deal with these subcases first. If $\la$ is a rectangle then it labels a finite-dimensional representation by Theorem \ref{unitary+fd}, hence $\Ores^n_{n-1}L(\ul\la)=0$ by \cite{BE}. If $b_4=b_5$ and $\ct(b_2)\neq \ct(b_5)\mod e$ then we may apply Lemma \ref{easy i-res} to conclude that $\Ores^n_{n-1}L(\ul\la)$ is semisimple or $0$ (in fact, it is either irreducible or $0$ because $b_2^t$ is not good removable in $|\ul\la^t,\ul{s}\rangle$, as we see next). 

Now assume $b_4=b_5$, $\ct(b_2)= \ct(b_5)\mod e$, and $b_2\neq b_4$, that is, $\la$ has exactly two removable boxes and the box of smallest content in $\la$ is one of them. Since $\ct(b_2)=e-s$, in $\la^t$ the box $b_2^t$ has content $s-e$, so by Lemma \ref{min cancellation pair} it is not a good removable $s$-box because it cancels with the addable $s$-box in component $2$ of $|\ul\la^t,\ul{s}\rangle$. The box $b_5^t$ is also an $s$-box of $|\ul\la^t,\ul{s}\rangle$ by assumption. There is no addable $s$-box of larger charged content. Therefore $b_5^t$ is the good removable $s$-box of $|\ul\la^t,\ul{s}\rangle$, and $\tilde{e}_s(\ul\la^t)=\ul\la^t\setminus b_5^t$. It remains to show that $\Ores^n_{n-1}L(\ul\la)$ is irreducible, since a priori it could be an extension. We have $\Ores^n_{n-1}L(\ul\la)\twoheadrightarrow  L((\tilde{e}_s(\ul\la^t))^t)=L(\ul\la\setminus b_5)$. By exactness of $\Ores^n_{n-1}$, then $\Ores^n_{n-1}\Delta(\ul\la)\twoheadrightarrow L(\ul\la\setminus b_5)$. Since $\la$ has no removable boxes besides $b_2$ and $b_5$, $[\Ores^n_{n-1}\Delta(\ul\la)]=[\Delta(\ul\la\setminus b_5)]+[\Delta(\ul\la\setminus b_2)]$ \cite{BE}. Observe that $c_{\ul\la\setminus b_5}>c_{\ul\la\setminus b_2}$. Therefore $[\Delta(\ul\la\setminus b_2):L(\ul\la\setminus b_5)]=0$, so $[\Ores^n_{n-1}\Delta(\ul\la):L(\ul\la\setminus b_5)]=[\Delta(\ul\la\setminus b_5):L(\ul\la\setminus b_5)]=1$. 
We deduce that $[\Ores^n_{n-1}L(\ul\la):L(\ul\la\setminus b_5)]=1$. Since the socle of $\Ores^n_{n-1}L(\ul\la)$ is equal to the head of $\Ores^n_{n-1}L(\ul\la)$, we conclude that $\Ores^n_{n-1}L(\ul\la)=L(\ul\la\setminus b_5)$.

Next, we consider the case $b_4\neq b_5$. Then in addition to $\ct(b_2)=e-s$ it is required that $\ct(b_3)-\ct(b_5)+1\leq e$ or $e=\ell-\ct(b_5)+1$ for some $b_4\leq \ell <b_3$. These two alternative conditions are equivalent to a single condition: $\ct(b_4)-\ct(b_5)+1\leq e$. Let $R$ be the set of removable boxes of $\la$. It follows that for all $b\neq b'\in R\setminus b_2$, $\ct(b)\neq \ct(b')\mod e$. Then we argue as in Lemma \ref{easy i-res} that $\EE_iL(\ul\la)$ is irreducible or $0$ for all $i\neq s\mod e$.  If some box $b\in R$ has the same content mod $e$ as $b_2$, then we argue as in the previous paragraph that $\EE_sL(\ul\la)=L(\ul\la\setminus b)$.\\


\noindent \ul{ \em Case(ii): $\ul\la=(\la^1,\la^2)$ with both $\la^1,\la^2\neq \emptyset$.}\\
As in  Proposition \ref{crystal ops unitaries}, it is sufficient to check the conditions for unitarity in cases (a), (b), and (f) of \cite[Corollary 8.5]{Griffeth}. For any $\ul\la\in\cP^2$, it holds that $\ct(b_4)\leq \ct(b_2)\leq \ct(b_1)$ and $\ct(b_4')\leq \ct(b_2')\leq \ct(b_1')$. In all three cases (a), (b), and (f) we then have $\ct(b_2)-\ct(b_4')+1\leq e-s$ and $\ct(b_2')-\ct(b_4)+1\leq s$. It follows that $-\ct(b_2)\leq -\ct(b_4)<-\ct(b_2')+s\leq -\ct(b_4')+s<-\ct(b_2)+e$. Therefore each $i\in\mathbb{Z}/e\mathbb{Z}$ occurs at most once as $\cct(b)\mod e$ for a removable box $b$ of $|\ul\la^t,\ul{s}\rangle$. By Lemma \ref{easy i-res}, $\Ores^n_{n-1}L(\ul\la)$ is semisimple or $0$.
\end{proof}


\textbf{Acknowledgments.} Thanks to C. Bowman, P. Etingof, T. Gerber, S. Griffeth, and J. Simental for useful discussions. I would like to especially thank J. Simental for bringing to my attention S. Montarani's work in \cite{Montarani}. Thank you to the anonymous referee for helping to get the baby version of this paper in order. Part of this work was carried out under the auspices of the grant SFB-TRR 195.

\end{document}